\documentclass[10pt,notitlepage]{article}
\usepackage{amsmath, amsthm, amssymb}
\usepackage{graphicx}%authblk
\usepackage[v2]{xy}
\xyoption{all}
\newtheorem{corollary}{Corollary}
\newtheorem{proposition}{Proposition}
\newtheorem{lemma}{Lemma}
\newtheorem{theorem}{Theorem}
\newtheorem*{thm}{Theorem}
\newtheorem{definition}{Definition}
\newtheorem{remark}{Remark}

\newcommand\ov{\overline}

\long\def\symbolfootnote[#1]#2{\begingroup%
\def\thefootnote{\fnsymbol{footnote}}\footnote[#1]{#2}\endgroup} 

\begin{document}
\title{{On the homotopy type of $A(\Sigma X)$}}
\date{}
\author{Crichton Ogle\\
Dept. of Mathematics\\The Ohio State University\\
\texttt{ogle@math.ohio-state.edu}}
\maketitle
\begin{abstract} We prove that the functor $X\mapsto A(\Sigma X)$ from connected pointed spaces to spectra, given by Waldhausen's $K$-theory, splits as a product of its Goodwillie derivatives.
\end{abstract}

\tableofcontents

\symbolfootnote[0]{2010 {\it Mathematics Subject Classification}. Primary 19D10; Secondary 57N37.}
\symbolfootnote[0]{{\it Key words and phrases}. Walhausen $K$-theory, Goodwillie Calculus.}
\vskip.5in

\section{Introduction/Statement of results}
\vskip.4in
In \cite{w2}, Waldhausen constructs a map 
$$
W:A(X)\to\Omega^\infty\Sigma^\infty(X_+)
$$ 
(where $A(X)$ denotes the
Waldhausen $K$-theory of the space $X$), and shows that evaluation on the image of $M : \Omega^\infty\Sigma^\infty(X_+)\to A(X)$ induced by the inclusion of monomial matrices produces a self-map $W\circ M : \Omega^\infty\Sigma^\infty(X_+) \to \Omega^\infty\Sigma^\infty(X_+)$ homotopic to the identity by a homotopy natural in $X$.  This yields a splitting of $\Omega^\infty\Sigma^\infty(X_+)$ off of $A(X)$ (as well as its stabilization $A^S(X)$), a fact which plays a key role in the proof of the
Fundamental Theorem of Waldhausen $K$-theory relating $A(X)$ to pseudo-isotopy theory (\cite{w2}, \cite{wm}, \cite{w} and finally \cite{jrw}): 

\begin{thm}[Waldhausen, Jahren, Rognes]
$A(X) \simeq \Omega^\infty\Sigma^\infty(X_+) \times Wh^{Diff}(X)$ where $\Omega^2 Wh^{Diff}(X)
\simeq \wp(X)$ = the stable pseudo-isotopy space of $X$ (as defined by 
Hatcher-Wagoner-Igusa).
\end{thm}

The construction of $W$ is in stages.  It is first shown that the homotopy fibre
\[
\ov{A}(S^n \wedge X_+) := hofibre(A(S^n \wedge X_+) \to A(*))
\]
can be described through a certain range of dimensions (approximately $2n$) in terms of a type of cyclic bar construction.  On this cyclic bar construction Waldhausen defines a map to $\Omega^\infty\Sigma^\infty(X_+)$ compatible with stabilization.  The result is a map $A^S(X) \to \Omega^\infty\Sigma^\infty(X_+)$ natural in $X$, and precomposition with the stabilization map $A(X) \to A^S(X)$ yields $W$.  In this paper we present a generalization of Waldhausen's map $W$.  Specifically let $X$ and $Y$ be pointed simplicial sets, $X$ connected.  Then there exists a \underbar{generalized Waldhausen trace map}:
\[
\overline{Tr}_X(Y) :  \underset{n\to\infty}{\lim}\Omega^n hofibre(\ov{A}(\Sigma(X \wedge \Sigma^n Y)) \to \overline A(\Sigma X)) \to \Omega^\infty\Sigma^\infty(\Sigma(\underset{q\ge 1}{\vee} | X^{[q-1]} \wedge Y |)).
\]

This map is natural in $X$ and $Y$.  The main application is to prove
\vskip.2in

{\bf\underbar{Theorem A}} {\it For connected $X$ there is a weak equivalence of infinite loop spaces
$$
\tilde\rho = \underset{q\ge 1}{\prod} \tilde\rho_q :
\Omega^\infty\Sigma^\infty(\Sigma(\underset{q\ge 1}{\vee} E\mathbb Z/q
\underset{\mathbb Z/q}{\leftthreetimes} |X|^{[q]})) \overset\simeq{\longrightarrow}
\overline A(\Sigma X) 
$$
natural in $X$, where the action of $\Bbb Z/q$ on $|X|^{[q]}$ is given by cyclic permutation.}
\vskip.2in

It should be noted that an alternative approach to the $p$-adic completion of this result would be to use the main result of [6A], which determines $\ov{A}(X)^{\wedge}_p$ for arbitrary 1-connected spaces (not just suspension spaces). Theorem A is, however, an integral identification of $\ov{A}(\Sigma X)$ obtained by explicit maps in a stable range.
\vskip.2in

Theorem A has an unfortunate history. In 1986, this result was announced simultaneously by the author and the authors of \cite{ccgh}, as a solution to a conjecture posed by Goodwillie a few years prior. Unfortunately both the unpublished \cite{o1}, \cite{o2}, as well as \cite{ccgh}, included a number of technical errrors. The proof of Theorem A we give here begins with the line of argument attempted in \cite{ccgh}, then modified along the lines of \cite{w2}.  An outline is as follows. In section 1 we recall the necessary results from \cite{w2} and Goodwillie's Calculus of Functors (\cite{g1} -- \cite{g3}), and define the maps $\tilde\rho_q$ used in the proof.  In section 2, we follow the arguments of \cite{w2} in constructing the trace map $\overline{Tr}_X(Y)$, culminating in section 2.3 where  we use $\overline{Tr}_X(Y)$ to explicitly compute the first derivative of $\tilde\rho_q$ at a connected space $X$ (in the sense of Goodwillie). Appealing to Goodwillie's convergence criteria for functor calculus \cite{g2}, together with his computation of the first differential of $\wp (X)$ \cite{g1} and the Fundamental Theorem of \cite{jrw} cited above then completes the proof of the result. It is quite likely that an integral statement similar to the above could alternatively be established by the methods of \cite{dgm}.
\vskip.2in

This paper represents work done by the author during the period 1988 - 1992, and as such represents a \lq\lq classical\rq\rq\ approach to the construction of trace maps on Waldhausen $K$-theory. During that period, the author benefitted from discussions with a number of mathematicians, including R. Schw\"anzl and R. Vogt (who carefully critiqued \cite{o1} and \cite{o2}), and Z. Fiedorowicz (with whom we were able to resolve some of the technical issues of these earlier attempts in \cite{fov}). But the {\it sine qua non} in all of this is the work of Waldhausen and Goodwillie; Waldausen for his development of the $K$-theory which bears his name and his contribution to the proof of the Fundamental Theorem stated above, and Goodwillie for his invention of the Calculus of Functors and the computation of the derivatives of stable pseudo-isotopy. What appears below is simply a corollary of their foundational work.
\vskip.2in

During the initial stages of this paper, the author enjoyed the hospitality of the Max-Plank-Institut f\"ur Mathematik in Bonn. This work partially supported by a grant from the National Science Foundation. We would like to thank Tom Goodwillie for his critical reading of earlier versions of this work, and the suggested improvements that followed. We would also like to thank the referee for the final version of this paper for his helpful and illuminating critque.
\newpage

%%%%%%%%%%%%%%%%%%%%%%%%%%%%%%%%%%%%%%%%%%%%%%%%%%%%%%%%%
%%%%%%%%%%%%%%%%%%%%%%%%%%%%%%%%%%%%%%%%%%%%%%%%%%%%%%%%%

\section{Background}

%%%%%%%%%%%%%%%%%%%%%%%%%%%%%%%%%%%%%%%%%%%%%%%%%%%%%%%%%

\subsection{Waldhausen $K$-theory} We recall the construction of $A(X)$ as given in \cite{w2}.  Let $X$ be a pointed, connected simplicial set, $GX$ its Kan loop group.  Let $H^n_k (|GX|)$ denote the total singular complex of the topological monoid $Aut_{|GX|} (\overset {k}{\vee}\ S^n \wedge |GX|_+)$ of $|GX|$-equivariant self-homotopy equivalences of the free basepointed $|GX|$-space $(\overset {k}{\vee}\ S^n \wedge |GX|_+)$. $H^n_k(|M|)$ is a mapping space (for a simplicial monoid $M$), for which we will adopt the convention that $AB = B\circ A$.
$H^n_k (|GX|)$ identifies naturally with a set $\overline M^n_k (|GX|_+)$ of path components of
\[
M^n_k (|GX|_+) = Map (\overset k {\vee} \ S^n,\overset k {\vee}\ S^n \wedge|GX|_+)
\]
under the inclusion 
\[
H^n_k (|GX|) \hookrightarrow Map_{|GX|} (\overset k {\vee} S^n \wedge |GX|_+, 
\overset k {\vee} S^n \wedge |GX|_+) \cong
Map (\overset k {\vee} S^n, \overset k {\vee} S^n\wedge |GX|_+)\; .
\]
One has stabilization maps 
\[
M^n_k (|GX|_+)\overset\imath{\longrightarrow} M^n_{k+1} (|GX|_+)
\]
given by wedge product with the inclusion  $S^n \hookrightarrow S^n \wedge |GX|_+$, suspension maps $M^n_k (|GX _+)\overset\Sigma{\longrightarrow} M^{n+1}_k (|GX|_+)$ given by smash product with the identity map on $S^1$, as well as pairing maps 
\[
M^n_k (|GX|_+)\times M_\ell^n (|GX|_+) \to M^n_{k+\ell} (|GX|_+)
\]
induced by wedge-sum.  Restricted to $\{H^n_k (|GX|)\}_{k\ge 0}$, this pairing gives $\underset{k\ge0}{\coprod} H^n_k (|GX|)$ the structure of a topological permutative category \cite{s1} for all $n\ge 0$.  These operations -- wedge sum, suspension, stabilization -- commute up to natural isomorphism.  So letting $H_k (|GX|) = \underset{\Sigma}{\underset{\to}{\lim}} H^n_k (|GX|)$ and $H(|GX|) = \underset{k}{\underset{\to}{\lim}} H_k (|GX|)$ we see
that $\underset{k\ge0}{\coprod} H_k (|GX|)$ is also a topological permutative category under wedge-sum.  Waldhausen's definition of $A(X)$ is

\begin{definition} If $X$ is a pointed connected simplicial set, $A(X) = \Omega B (\underset{k\ge 0}{\coprod}  BH_k (|GX|)) \simeq \Bbb Z \times BH (|GX|)^+$
\end{definition}

where \lq\lq $\simeq$ \rq\rq\ means \lq\lq weakly equivalent \rq\rq. If $X$ is a connected basepointed space, $A(X)$ is defined to be $A(Sing(X))$ (here $Sing(X)$ denotes the basepointed total singular complex of $X$). Similarly if $X$ is a connected pointed simplicial space, $A(X) \overset{\text{def}}{=} A(Sing|X|)$. $A(X)$ is a homotopy functor, in the sense that $X\simeq Y$ implies  $A(X)\simeq A(Y)$.
\vskip.2in

We will use the notation $\Sigma U$ to denote the reduced suspension of $U$.  If $|X| \simeq\Sigma|Z|$, where $Z$ is a simplicial space connected in each degree, then $GX$ is weakly equivalent to the simplicial James monoid $JZ$, which in degree $q$ is the free monoid on the pointed space $Z_q$.  In this case we can use $JZ$ in place of the Kan loop group $GX$ in the above constructions. The result is an equivalence $A(\Sigma Z) \simeq \Omega B (\underset{k\ge0}{\coprod} BH_k (|JZ|))$.

In studying $A(\Sigma Z)$, we will use constructions from \S2 of \cite{w2}.  The first, due to Segal, generalizes the bar construction which associates to a monoid its nerve.  Thus, a \underbar{partial monoid} is a basepointed set $M$ together with a partially defined composition law 
\[
M\times M\supset M_2 \overset\mu{\hookrightarrow}  M
\]
satisfying
\begin{itemize}
\item $M\vee M \subset M_2$, with $\mu(*,m) = \mu(m,*) = m$ and 
\item $\mu(\mu(m_1,m_2), m_3) = \mu(m_1,\mu(m_2,m_3))$ in the sense that if one is defined then so is the other, and they are equal.
\end{itemize}
To avoid confusion, we include as part of the data associated to a partial monoid the sets
\[
\{ M_p\subseteq\text{ composable }p\text{-tuples in }M\}_{p\ge 0}
\]
where $M_0 = *$ , $M_1 = M$ , $M_2$ is as given above, and \lq\lq composable\rq\rq\ means that the iterated product is defined. We require that the face and degeneracy maps, defined in the usual way, induce a simplicial structure on $\{M_p\}_{p\ge 0}$ . The \underbar{nerve} of $M$ is then the simplicial set
\[
\{[p]\mapsto M_p\}.
\]

Let $M$ be a monoid, $S$ a set on which $M$ acts on both sides, with $(ms)m' = m(sm')$ .One can then form the \underbar{cyclic bar construction} of $M$ with coefficients in $S$.  It is a simplicial set $N^{cy}(M,S)$ which in degree $q$ is $M^q\times S$.  The face and degeneracy maps are given by the following formulae (see \cite{w2}, \S2):

\begin{align*}
\partial_0(m_1,\dots,m_q;s) &= (m_2,m_3,\dots,m_q;sm_1)\\
\partial_i(m_1,\dots,m_q;s) &= (m_1,\dots,m_i m_{i+1},\dots, m_q;s),\qquad
 1\le i \le q-1\\
\partial_q(m_1,\dots,m_q;s) &= (m_1,\dots,m_{q-1};m_q s)\\
s_i(m_1,\dots,m_q;s) &= (m_1,\dots,m_i,1,m_{i+1},\dots;s), \qquad 
0\le i \le q.
\end{align*}

As noted in \cite{w2}, the double bar construction is a special case of this cyclic bar construction where $S$ appears as a cartesian product of a left $M$-set and a right $M$-set.  When $M$ is a grouplike monoid ($\pi_0M$ is a group) and $S=M$ with induced $M$ action on the left and right, there is a weak equivalence between $N^{cy}(M,M)$ and $BM^{S^1} = $ the free loop-space of $BM$.  The construction of $N^{cy}(M,S)$ extends in the obvious way to a simplicial monoid $M$ acting on a simplicial set $S$.
\vskip.2in

It is sometimes the case that $S$ itself is a partial monoid which admits a left and right $M$-action.  In this case one wants to know that the cyclic bar construction $N^{cy}(M,S)$ can be done in such a way as to be compatible with the partial monoid structure on $S$. So let $M$ be a monoid.  By a \underbar{left $M$-monoid} we will mean a partial monoid $E$  together with a basepointed $M$-action $M\times E\to E$ compatible with the partial monoid structure on $E$ in the sense that for each $m\in M$ the map
\begin{gather*}
E\to E\;;\qquad e\mapsto me 
\end{gather*}
is a homomorphisms of partial monoids.  A \underbar{right} $M$-monoid is similarly defined, and an \underbar{$M$-bimonoid} is a partial monoid equipped with compatible left and right monoid structures.  Given such an $M$-bimonoid $E$ satisfying an obvious ``saturation''\ condition [p. 367, W2], the semidirect product $M\ltimes E$ is the partial monoid with composition given by $(m,e) (m',e') = (mm', (em') (me'))$  whose nerve is the simplicial set
\[
\{[p]\mapsto M^p\times E_p\}_{p\ge 0},
\]
where $E_p$ is the set of composeable $p$-tuples (as defined above).  When $E$ is commutative, we will write the product in $E$ additively. Clearly this construction can be done degreewise when $M$ and $E$ are simplicial.  If the partial monoid structure on $E$ has not been specified, we will assume it is the trivial one, for which $E_p = \overset p{\vee} (E,*)$ .  The simplicial set $\{[p]\mapsto \overset p{\vee} (E,*)\}$ is a left (resp.\ right resp.\ bi-) monoid over $M$ if $E$ is.  Iteration of this construction yields an $M$-monoid structure on a space whose realization is an iterated suspension of $|E|$, and which agrees with that induced by the given action of $M$ on $E$ together with the trivial action on the suspension coordinates. A left (resp\@. right) \underbar{$M$-space} is a space equipped with a left (resp\@. right) action of M, not necessarily basepointed. Obviously, a left (resp\@. right) $M$-monoid is a left (resp\@. right) $M$-space, though not conversely. Similarly, we may define an $M$-bispace to be one equipped with  compatible left and right $M$ actions.
 \vskip.2in

Suppose that $A$ is a monoid, and that we are given an inclusion $A\hookrightarrow M$ of $A$  into an $A$-bispace which is a morphism of $A$-bispaces. Identifying $A$ with its image in $M$ , one may define a partial monoid by
\[
M_p := \overset p{\vee} (M,A) = \overset
p{\underset{j=1}{\cup}} A^{j-1} \times M \times A^{p-j}.
\]
The nerve of the resulting partial monoid $\{[p] \mapsto M_p\}$ is referred to as a \underbar{generalized wedge} (this construction is due to Waldhausen). Taking $A = \{pt\}$ yields $\{[p] \mapsto \overset p{\vee} (M,*)\}$ whose realization is
homeomorphic to $\Sigma|M|$ (as discussed above).  It is often useful to approximate the nerve of a monoid $M$ by generalized wedges.  A straightforward argument [Lemma 2.2.1, W2] shows that if $A\hookrightarrow M$ is an $(n-1)$-connected inclusion of monoids, the induced inclusion 
\[
\{[p]\mapsto \overset p{\vee} (M,A)\} \hookrightarrow \{[p] \mapsto \overset p{\vee}
(M,M)\} = NM
\]
is $(2n-1)$-connected.  As one can easily see, a fixed monoid may admit many different partial monoid structures. A key result concerning the nerve of a semidirect product is provided by [Lemma 2.3.1, W2]; it provides a map 
\[
u : diag (N^{cy} (M,NE))\to N(M\ltimes E)
\]

which is a weak equivalence when $\pi_0(M)$ is a group. Here $M$ is a simplicial monoid, $E$ a simplicial $M$-bimonoid, and $N^{cy}(M,NE)$ the cyclic bar construction of $M$ acting on the nerve of the partial monoid $E$. The ``diagonal'' structure is with respect to the simplicial coordinates coming from $N^{cy}(_-)$ and $NE$. The saturation condition referred to above, as well as the condition that $\pi_0(M)$ is a group, will always be satisfied in our case.  As we will need to know $u$ explicitly later on, we recall that it is given on $n$-simplices by the formula [p.\ 369, W2]:
\begin{equation}\label{eqn:1.1.2}
u(m_1,\dots,m_n;e_1,\dots,e_n) 
= (m_1,(\overset
n{\underset{i=1}{\prod}} m_i)e_1 m_1; m_2,(\overset
n{\underset{i=2}{\prod}} m_i) e_2(m_1 m_2);\cdots; m_n, m_n e_n
(\overset n{\underset{i=1}{\prod}} m_i)).
\end{equation}

Let us return to considering $JZ$ and $H^n_k(|JZ|)$ (for connected $Z$).  We will be interested in the case when $Z = X\vee Y$. For a positive integer r, let $F_r(X,Y) \subset J(X\vee Y)$ denote the subset which in each degree consists of elements of word-length at most $r$ in $Y$.  This is clearly a simplicial subset. There is also a natural partial monoid structure on $F_r(X,Y)$, where two elements are composable if their product in $J(X\vee Y)$ lies in $F_r(X,Y)$. We will denote $F_1(X,Y)/F_0(X,Y)$ by $\overline F_1(X,Y)$. Note that $\overline F_1(X,Y) \cong J(X)_+\wedge Y\wedge J(X)_+$ for connected $Y$. Now consider the projection maps
\begin{gather}
F_1(X,Y)\twoheadrightarrow F_1(X,*)= JX\;\;
\text{induced by}\;\;Y\twoheadrightarrow *, \label{eqn:f1}\\
F_1(X,Y) \twoheadrightarrow \overline F_1(X,Y).\label{eqn:f2}
\end{gather}

We observe that all spaces in sight admit compatible left and right $J(X)$-actions, and these maps commute with the action. Also both projection maps are partial monoid maps, where we take the trivial monoid structure on $\overline F_1(X,Y)$, which with the given left and right actions of $J(X)$ is a $J(X)$- bimonoid. Let
\[
\overline M^n_k(|F_1(X,Y)|_+) = \overline M^n_k(|J(X,Y)|_+) \cap M^n_k(|F_1(X,Y)|_+).
\]

\begin{lemma} \label{lemma:p1p2}
The projection maps in (\ref{eqn:f1}), (\ref{eqn:f2})  induce maps
\begin{align*}
&\overline M_k^n(|F_1(X,Y)|_+) \overset{p_1}{\twoheadrightarrow}
\overline M_k^n(|J(X)|_+)\cong H_k^n(|J(X)|), \\
&\overline M_k^n(|F_1(X,Y)|_+) \overset{p_2}{\twoheadrightarrow}
M_k^n(|\overline F_1(X,Y)|).
\end{align*}
All four spaces admit compatible left and right actions of $H^n_k(|J(X)|)$, and these maps commute with the actions. As a result, these projection maps induce a map of generalized wedges
\begin{align*}
\{[p] &\mapsto \overset {p}{\vee}\; (\overline M^n_k (|F_1(X,Y)|_+), H^n_k(|JX|))\}\\
\to \{[p] &\mapsto \overset {p}{\vee}\; (H^n_k(|JX|) \ltimes M^n_k (|\overline F_1(X,Y)|), H^n_k(|JX|))\}
\end{align*}
where the semi-direct product $H^n_k(|J(X)|) \ltimes M^n_k(|\overline F_1(X,Y)|)$ is formed using the trivial monoid structure on $M^n_k(|\overline F_1(X,Y)|)$ together with the given left and right (basepointed) $H^n_k(|J(X)|)$-actions.
\end{lemma}

\begin{proof} The left and right actions of $H^n_k(|J(X)|)$ on $H^n_k(|J(X \vee Y)|)$ induce compatible left and right actions on $\overline M^n_k(|F_1(X,Y)|_+)$. These actions are functorial in $Y$, hence natural with respect to the map $Y\twoheadrightarrow *$ inducing the projection map $p_1$. It is not hard to see that these actions are also compatible with the collapsing map which induces $p_2$ . It follows that $p_1$ and $p_2$, taken together, induce a map which on the level of sets is represented as
\begin{equation}\label{eqn:p1p2}
\overline M_k^n(|F_1(X,Y)|_+) \overset{p_1 \times p_2}{\longrightarrow} H_k^n(|J(X)|)\times M_k^n(|\overline F_1(X,Y)|). 
\end{equation}

Taking the trivial monoid structure on $M^n_k(|\overline F_1(X,Y)|)$, we get an $\overline H^n_k(|J(X)|)$ - bimonoid structure on this space, so that the R.H.S. of (\ref{eqn:p1p2}) is the underlying set of a semi-direct product.  The partial composition law in this semi-direct product amounts to a description of a compatible left and right action of $H^n_k(|J(X)|)$. This action is given explicitly by
\begin{align*}
x(y,a) &= (x,*)(y,a) = (xy,xa),\\
(y,a)x &= (yx,ax),\\
x,y &\in H_k^n(|J(X)|),\;\; a \in M_k^n(|\overline F_1(X,Y)|).
\end{align*}

The projection maps $p_1$ and $p_2$ preserve both the left and right actions of $H^n_k(|J(X)|)$, and therefore so does $p_1\times p_2$ under the actions described above. This implies that $p_1\times p_2$ induces a map of generalized wedges, as claimed.
\end{proof}
\vskip.5in

%%%%%%%%%%%%%%%%%%%%%%%%%%%%%%%%%%%%%%%%%%%%%%%%%%%%%%%%%

\subsection{Goodwillie Calculus} This section recalls results from Goodwillie's Calculus ([G1] -- [G3]) that will be used later on. We consider functors $F : \underline C \to \underline D$ , where $\underline C$ is either $U,T,U(C)$ or $T(C)$ and $\underline D$ is $T,T(C)$ or the category $Sp$ of spectra.  Here $U$ is the category of (Hausdorff) topological spaces, $T$ the category of basepointed spaces in $U$.
$U(C)$ denotes the corresponding category of spaces over  $C\in obj(U)$, and $T(C)$ the category of basepointed objects in $U(C)$.  Note that an object of $T(C)$ is a retractive space $Y$ over $C$, i\@.e\@., $r :Y\to C$ comes equipped with a right inverse  $i\;(r \circ i = id)$.  Each of these choices of $\underline C$ is a  closed model category in the sense of Quillen, so one has the usual constructions of homotopy theory. When $\underline C$ = $U(C)$ or $T(C)$ we denote by $C_n$ the full subcategory of $n$-connected objects in $\underline C$. 
As in [G2], a map of spaces or spectra is called \underbar{$n$-connected} if its homotopy fibre is $(n-1)$-connected. In all of these categories one has a standard notion of weak equivalence, and $F$ is called a \underbar{homotopy functor} if $F$ preserves weak equivalences. We will only be concerned with homotopy functors.
\vskip.2in

Let $S$ be a finite set, $C(S)$ the category of subsets of $S$ with morphisms corresponding to inclusions.  An $S$-cube in $C$ is a covariant functor $G : C(S) \to \underline C$.  If $S = \{1,2\dots,n\} = \underline n\,, G$ is called an $n$-cube.  Associated to an $S$-cube is the homotopy-inverse limit $h(G) = holim(G|_{C_0(S)})$ where $C_0(S)$ denotes the full subcategory of $C(S)$ on all objects except $\phi$.  The natural coaugmentation map $lim(G) \to holim(G)$ induces a natural transformation 
\[
a(G) : G(\phi) \to h(G)\; .  
\]
and $G$ is \underbar{$h$-cartesian} if $a(G)$ is a weak equivalence.  We say $F : \underline C \to\underline D$ (as above) is \underbar{$n$-excisive} if $F\circ G$ is $h$-cartesian for every strongly homotopy co-cartesian $S$-cube $G : S \to \underline C$, where $|S| = n+1$ [Def.3.1, G2].  The condition that $F$ be $n$-excisive becomes less restrictive as $n$ increases. Thus, if $F$ is $n$-excisive, it is $(n+1)$-excisive, but not conversely [Prop\@. 2.3.2, G2].
\vskip.2in

Given a homotopy functor $F$ satisfying certain conditions, there is a natural way of producing a functor $P_n F$ of degree $n$ and a natural transformation $F\to P_n F$.  In fact, $P_n F$ can always be constructed.  Starting with $X\in \text{ obj }(\underline C)$ one can
define an $(n+1)$-cube 
\[
X\underset C {*}(_-) : C (\underline{n+1}) \to \underline C = U(C)\text{ or }T(C);
\]
this associates to $T \subset \underline{n+1}$ the space $X\underset{C}{*}T$ which is the fibrewise join over $C$ of $X$ with the set $T$.  Now let $(T_n F) (X) = holim (F\circ (X\underset{C}{*}(_-))|_{C_0(\underline{n+1})})$. Then $a(F\circ(X\underset{C}{*}(_-)))$ defines a transformation $(t_n F) (X) : F(X) \to (T_n F) (X)$.  One easily sees that $X\mapsto (T_n F) (X)$ is again a homotopy functor on $\underline C$ and that $(t_n F) := (t_n F) (_-)$ defines a natural transformation from $F$ to $T_n F$.  Note that $X\underset{C}{*}(_-) :
C(\underline{n+1}) \to \underline C$ is a (strongly) homotopy co-cartesian diagram in $\underline C$, so that $t_n F$ is an equivalence if $F$ is of degree $n$.  Iteration of this construction yields $P_n F$ which is by definition the homotopy colimit of the directed system $\{T^i_n F, t_n T_n^i F\}$. The transformations $t_n T^i_n F$ induce a natural transformation $p_n F : F \to P_n F$.  Choice of a distinguished element $(m+1) \in \underline{m+1}$ induces an inclusion $\underline m \to \underline{m+1}$ and hence a natural transformation $C( \underline{m})\to  C(\underline{m+1})$.  This in turn induces a natural transformation of directed systems 
\[
\{T^i_n F, t_n T_n^i F\} \to \{T^i_{n-1} F,t_{n-1} T^i_{n-1} F\}
\]
and hence a natural transformation  $P_m F\overset{q_m F}{\longrightarrow} P_{m-1} F$.  Different choices of $m$ yield naturally equivalent choices of $q_m F$.  The \underbar{Goodwillie Taylor series of $F$} is then by definition the inverse system $\{P_n F, q_n F\}$, which is best viewed as a tower together with the natural transformations $p_n F$:

\[
\diagram
 & & & \vdots \dto^{q_3F} \\
 & & & P_2F \dto^{q_2F} \\
 & & & P_1F \dto^{q_1F}  \\
 F \xto[0,3]^{p_0F} \xto[-1,3]^{p_1F} \xto[-2,3]^{p_2F} & & & P_0 F
\enddiagram
\]

The closed diagrams in this tower are homotopy commutative.  The $\text{n}^{\text{th}}$ homogeneous part of $F$ is by definition the homotopy fibre of $q_nF$:
\[
D_n F := hofibre(P_n F \overset{q_n F}{\longrightarrow} P_{n-1} F)\;.
\]
\vskip.2in

\begin{definition}([Definition 4.1, G2]) $F$ is \underbar{stably $n$-excisive} if the following statement holds for some numbers $c$ and 
$\kappa$ :\newline
\noindent $\underline{E_n(c,\kappa)}$ : If $G : C(\underline{n+1}) \to \underline{C}$ is any strongly co-Cartesian $(n+1)$-cube such that for all $s\in S$ the map $G(\phi)\to G(s)$ is $k_s$ connected and $k_s\ge \kappa$, then the diagram $G(C(\underline{n+1}))$ is $(-c + \sum k_s)$-connected.
\end{definition}

In this case $D_n F$ is \underbar{homogeneous} of degree $n$ that is, it is stably $n$-excisive and $P_i D_n F \simeq *$ for $i<n$ [Prop. 1.11, G3].  We will write $\underline{P_n F}$ for $hofibre(F\overset{p_n F}{\longrightarrow} P_n F)$, and $\underline{P^m_n F}$ for $hofibre(P_n F \to P_m F)$. The functor $D_n(F)$ is referred to as the \underbar{$\text{n}^{\text{th}}$ differential} of $F$ (at $*$). One also wants to know when the connectivity of $F\overset{p_n F}{\longrightarrow} P_n F$ tend to $\infty$ as $n$ tends to $\infty$.  From [Def. 4.2, G2], $F$ is $\rho$-analytic if there is some number $q$ such that $F$ satisfies $E_n(n\rho -q, \rho +1)$ for all $n\ge 1$ .

\begin{thm} ([Th. 2.5.21, G3]) The connectivity of $p_n F$ tends to $\infty$ over the category $\underline C_{\rho}$ where $\rho = \rho(F), F : \underline C \to \underline D$.
\end{thm}

In analogy with functions, $\underline C_\rho$ may be thought of as the \underbar{disk of convergence} of $F$.  In applying this calculus to $F$, it is natural to restrict one's attention to the subcategory $\underline C_{\rho(F)}$ which in general is the largest subcategory of $\underline C$ for which the Taylor series of $F|_{\underline C_{\rho(F)}}$ converges (in the homotopy-theoretical sense).  Within this range it provides a powerful machinery for analyzing $F$, as well as determining the effect of a natural transformation $\eta :F_1 \to F_2$ on homotopy groups.  It is clear from the above theorem that $\eta$ will induce a weak equivalence when restricted to $\underline C_\rho \;(\rho = max(\rho(F_1),\rho(F_2)), F_i : \underline C \to \underline D)$ if $\eta$ induces an equivalence on differentials: 
\[
D_n(\eta) : D_n(F_1)\overset{\cong}{\to} D_n(F_2),
\]
under the condition that $P_0(F_i) \simeq *$.  However, there is another way of getting at $\eta$.  Assume first that $\underline C = U(C)$ and that $F_i : \underline C \to \underline D$ have the same modulus $\rho$ for $i$ = 1,2.  Let $(X,p : X\to C)$ be an object in $U(C)$.  Then $(X,p : X\to C)$ defines a natural transformation $\nu_{(X,p)} : U(X) \to U(C)$ given on objects by
\[
\nu_{(X,p)}(Y,r : Y\to X) = (Y,p \circ r : Y \to C).
\]
Analyticity is preserved by the natural transformation $\nu^*_{(X,p)} : F\to F\circ \nu_{(X,p)}$.  The next result of Goodwillie's concerns only $1^{\text{st}}$ differentials, and is is contained in Theorems 5.3 and 5.7 of [G2].

\begin{theorem}\label{thm:conv} If $F_1,F_2: U(C) \to \underline D$ are $\rho$-analytic, and $\eta : F_1 \to F_2$ is a natural transformation such that the square 
\[
\diagram
P_1(\nu^*_{(X,p)}F_1) \rrto^{P_1(\nu^*_{(X,p)}\eta)} 
 \dto_{q_1(\nu^*_{(X,p)}F_1)}  & & P_1(\nu^*_{(X,p)}F_2)
                      \dto^{q_1(\nu^*_{(X,p)}F_2)}  \\
P_0(\nu^*_{(X,p)}F_1) \rrto^{P_0(\nu^*_{(X,p)}\eta)}  & & P_0(\nu^*_{(X,p)}F_2)
\enddiagram
\]
is homotopy-cartesian for every $(X,p)$ in $U(C)$, then for every $f : Y\to X$ in $U(C)_\rho$ the diagram
\[
\diagram
F_1(Y) \rto^{\eta(Y)} \dto_{F_1(f)}  & F_2(Y) \dto^{F_2(f)}  \\
F_1(X) \rto^{\eta(X)}  & F_2(X)
\enddiagram
\]
is homotopy-cartesian.
\end{theorem}

In the case $C = *$ we will denote $hofibre(q_1(\nu^*_{(X,p)} F))$ by $(D_1F)_X$; $p$ in this case is unique.  For $F_2 = A(\Sigma_-) =
A\Sigma(_-)$, $\rho =0$ by [Theorem 4.6, G2].  Theorem \ref{thm:conv} yields

\begin{corollary} If $\eta : F_1 \to A\Sigma(_-)$ is a natural transformation which induces an equivalence $D_1(\eta)_X : (D_1 F_1)_X \overset\simeq{\longrightarrow} (D_1 A\Sigma)_X$ for all connected spaces $X$ and $F$ is $0$-analytic, then $\eta$ induces an equivalence 
\[
\eta(f) : hofibre(F_1(Y) \to F_1(X)) \overset\simeq{\longrightarrow} hofibre(A(\Sigma Y) \to A(\Sigma X)) 
\] 
for all maps $f$ between connected spaces $Y$ and $X$.
\end{corollary}

The result which makes these techniques applicable to the study of $A(X)$ is the computation, due to Waldhausen at $X=pt$,  and Goodwillie for general $X$, of the differentials of $A(X);\text{ here } (Y)$ denotes the retractive object $(Y\vee X, r : Y\vee X\to X)$ thought of as an object in $T(X)$.

\begin{theorem} (Waldhausen [W2], [W4]; Goodwillie [Cor. 3.3, G1])\label{thm:WG} For connected $X$ there is an equivalence
\[
(D_1A\Sigma)_X(Y)\simeq\Omega^\infty\Sigma^\infty(\Sigma(\underset{q\ge1}{\vee} |X^{[q-1]} \wedge Y|))
\]
natural in $X$ .
\end{theorem}

\begin{proof} Goodwillie's computation in [G1] applies to the functor  $A(_-)$ rather than $A\Sigma(_-)$. However, there is a natural equivalence between $A\Sigma(_-)$ and the restriction of $A(_-)$ to the subcategory of (basepointed) suspension spaces. Again, by Goodwillie, the differential $(D_1A\Sigma)_X(Y)$  may be computed as the differential (at $\Sigma X$) of the functor
\[
(Y) \to  A(\Sigma(Y\vee X)\to \Sigma X) := hofibre(A(\Sigma Y\vee\Sigma X)\to A(\Sigma X))
\]
which by [G2], together with the Snaith splitting of the functor $\Omega\Sigma(_-)$ yields the result.
\end{proof}

For a functor $F$ with range either $T$, $T(C)$, or $Sp$, $F$ is said to be \underbar{continuous} if the natural map of $Hom$-spaces $Hom(X,Y)\to Hom(F(X),F(Y))$ is continuous, and \underbar{finitary} if it satisfies the colimit axiom [G3]. Waldhausen's functor $A(_-)$ is both continuous and finitary.

\begin{lemma} If $X = \{X_k\}_{k\ge 0}$ is a simplicial space and $X_k$ is 1-connected for each $k$, then
\[
\Phi_A : | [k] \mapsto \overline A(X_k) | \overset\simeq{\longrightarrow} \overline A(| [k] \mapsto X_k|).
\]
\end{lemma}
\begin{proof} By [BD], if $F$ is a continuous finitary homotopy functor on $U(C)$ then the natural transformation $\Phi_F : |[k] \mapsto F(_-)| \to F(|[k] \mapsto (_-)|)$ induces a weak equivalence over the category of simplicial objects in $U(C)_{\rho(F)}$.
\end{proof}
\vskip.5in

%%%%%%%%%%%%%%%%%%%%%%%%%%%%%%%%%%%%%%%%%%%%%%%%%%%%%%%%%

\subsection{Elementary Expansions and Representations in $H^n_q(|JX|)$}

As in the previous sections $X$ will denote a basepointed connected simplicial set.  Our objective in the section will be to construct the maps $\tilde\rho_q :\tilde D_q(X) \to \tilde A(\Sigma X)$ of \cite{ccgh} used in the proof of Theorem A, where
\[
D_q(X) := \Omega^\infty\Sigma^\infty(\Sigma( E\mathbb Z/p
\underset{\mathbb Z/p}{\leftthreetimes} |X|^{[p]})).
\]
In addition, we provide some techniques for computing $\tilde\rho_q$ on differentials, by relating certain restrictions of $\tilde\rho_q$ to products of elementary expansions. This will be used in section 3.3 where we compute the trace of $\tilde\rho_q$.  From the construction of $\tilde\rho_q$, it will be easy to extend it to a map $\tilde\rho_q(JX) : \tilde D_q(JX) \to \overline A(\Sigma X)$.  We do this, and prove analogous results for $\tilde\rho_q(JX)$.

Let $\imath : |X| \to |JX|$ denote the standard inclusion.  Fixing an indexing of $\overset q{\vee} S^n$ and $\overset q{\vee} (S^n \wedge
|JX|_+)$ we let $(S^n)_i$ resp\@. $(S^n \wedge |JX|_+)_i$ denote the $i\text{th}$ term in the appropriate wedge for $1\le i \le q$.  Beginning with $(x_1\cdots,x_q)\in |X|^q$, $\rho_q(x_1,\cdots,x_q)$ is the map given on $(S^n)_i$ by the composition 

\begin{equation}\label{eqn:elem1}
 (S^n)_i = S^n \xrightarrow{\text{pinch}} S^n \vee S^n
\overset{id\vee f_i}{\longrightarrow}
S^n\vee (S^n\wedge|JX|)\overset{inc}{\longrightarrow} 
(S^n\wedge |JX|_+)_i \vee (S^n\wedge |JX|_+)_{i+1}.
\end{equation}

Here subscripts are taken mod $q$; thus $i+1 = 1$ if $i = q, i+1$ otherwise.  The basepointed cofibration sequence 
\[
S^0 \overset i{\rightarrow} |J(X)|_+ \twoheadrightarrow |J(X)|
\]
splits up to homotopy after a single suspension. Fixing $j_1 : \Sigma|JX| \to \Sigma (|JX|_+)$ with $\Sigma p \circ j_1 \simeq id$ and letting $j : \Sigma^n|JX| \to \Sigma^n(|JX|_+)$ be $\Sigma^{n-1}(j_1)$, $inc$ is the map induced by the inclusions

\begin{align*}
&S^n = S^n \wedge S^0\hookrightarrow S^n\wedge |JX|_+\; , \\
&S^n \wedge |JX| \overset j{\hookrightarrow} S^n \wedge |JX|_+.
\end{align*}

Then $f_i(s) = [s,\imath(x_i)]\in S^n \wedge |JX|$ for $s\in S^n$. ``pinch'' denotes the pinch map determined by the standard embedding $S^{n-1} \to S^n$, together with a fixed choice of homeomorphism from the cofibre to $S^n\vee S^n$ . Clearly $\rho_q$ is continuous and defines a map of spaces
\[
\rho_q : |X|^q \to |H^n_q(|JX|)|\cong| \overline M^n_q(|JX|_+)|.
\]
$\rho_q$ is also equivariant with respect to the action of $\Bbb Z/q$, which acts on $|X|^q$ by cyclically permuting the coordinates and
on $H^n_q(|JX|)$ via the standard embedding $\Bbb Z/q \to \Sigma_q$ and the usual conjugation action of $\Sigma_q$ on $H^q_n(|JX|)$.

\begin{proposition}\label{prop:extend} $\rho_q$ extends to a map $\overline\rho_q : E\Bbb Z/q \underset{\Bbb Z/q}{\times} |X|^q \to \Omega\overline A(\Sigma X)$, which in turn induces a map $\tilde\rho_q : \Omega^\infty\Sigma^\infty(\Sigma (E\Bbb Z/q \underset{\Bbb Z/q}{\leftthreetimes} |X|^{[q]})) \to \overline A(\Sigma X)$.
\end{proposition}

\begin{proof} Taking the direct limit under suspension and stabilization yields a map $|X|^q \to |H(|JX|)|$ which we also denote by $\rho_q$.  This map is still $\Bbb Z/q$-equivariant, where $\Bbb Z/q$ acts on the second space via the embedding $\Bbb Z/q \to \Sigma_q \to \Sigma_\infty$.  It suffices to know now that the plus construction $|H(|JX|)| \to \Omega A(\Sigma X)$ can be done so as to be equivariant with respect to the action of $\Sigma_\infty$ and that the action of $\Sigma_\infty$ on $\Omega A(\Sigma X)$ is trivial up to homotopy.  This follows from \cite{fo}. The result is that $\Omega A(\Sigma X) \overset i{\rightarrow} \underset{\Sigma_\infty\phantom{xx}}{E\Sigma_\infty \times\Omega A(\Sigma X)}$ admits a left homotopy inverse 
\[
p : \underset{\Sigma_\infty\phantom{xx}}{E\Sigma_\infty \times\Omega A(\Sigma X)} \to \Omega A(\Sigma X)
\]
$(p\circ i\simeq id)$ and we can take $\overline\rho_q$ to be the composition 
\[ 
E\Bbb Z/q \underset{\Bbb Z/q}{\times} |X|^q
\overset{(1\times\rho_q)}{\longrightarrow} \underset{\Sigma_\infty\phantom{xxx}}{E\Sigma_\infty{\times} |H(|JX|)|} \to \underset{\Sigma_\infty\phantom{xx}}{E\Sigma_\infty \times\Omega A(\Sigma X)} \overset p{\rightarrow} \Omega A(\Sigma X).
\]
Taking the infinite-loop extension of the adjoint of $\overline\rho_q$ yields a map
\[
\Omega^\infty\Sigma^\infty(\Sigma(E\Bbb Z/q\underset{\Bbb Z/q}{\times} |X|^q)) \to A(\Sigma X).
\]
Fix a stable section $s$ of the projection
\[
E\Bbb Z/q\underset{\Bbb Z/q}{\times} |X|^q \to E\Bbb Z/q
\underset{\Bbb Z/q}{\leftthreetimes}|X|^{[q]} = (E\Bbb Z/q)_+\underset{\Bbb Z/q}{\wedge}|X|^{[q]}.
\]
Then $\tilde\rho_q$ is the composition
\[
\Omega^\infty\Sigma^\infty (\Sigma(E\Bbb Z/q\underset{\Bbb Z/q}{\leftthreetimes}|X|^{[q]})) \overset s{\rightarrow}
\Omega^\infty\Sigma^\infty (\Sigma(E\Bbb Z/q\underset{\Bbb Z/q}{\times}|X|^q)) \to
A(\Sigma X).
\]
Finally we note that all of the constructions are natural in $X$, and hence factor through $\overline A(\Sigma X)$.
\end{proof}

The space $\Omega^\infty\Sigma^\infty(\Sigma(E\Bbb Z/q\underset{\Bbb Z/q}{\leftthreetimes}|X|^{[q]}))$ will be denoted by $\tilde D_q(X)$.  $\tilde D_q(_-)$ can alternatively be thought of as a functor on connected spaces.  The following is more or less contained in [\S3, CCGH]. Here $X \text{ and } Y$ denote basepointed simplicial sets.

\begin{proposition}\label{prop:rep1}
\begin{enumerate}
\item   $(D_1\tilde D_q)_X(Y) \simeq \Omega^\infty\Sigma^\infty(\Sigma|X^{[q-1]} \wedge Y|)$.
\item $(D_1 F_q)_X(Y) =\Omega^\infty\Sigma^\infty(\Sigma(\overset q{\underset{i=1}{\vee}} |X^{[i-1]} \wedge Y \wedge X^{[q-i]}|))$, where $F_q(Z)=\Omega^\infty\Sigma^\infty(\Sigma|Z^{[q]}|)$.  The natural transformation $F_q(_-) \to\tilde D_q(_-)$ induces the fold map on $1^{\text{st}}$ differentials which is
the infinite loop extension of the map 
\begin{gather*}
\overset{q-1}{\underset{i=0}{\vee}} X^{[q-i-1]} \wedge Y \wedge X^{[i]} \to X^{[q-1]} \wedge Y,\\ 
(x_1, \cdots, x_{q-i-1}, y,x'_1,\cdots,x'_i)\mapsto (x'_1,\cdots,x'_i,x_1,\cdots, x_{q-i-1}, y).
\end{gather*}
\item The inclusion $i_q(X,Y) : X^{[q-1]} \wedge Y \to (X\vee Y)^{[q]} \to E\Bbb Z/q\underset{\Bbb Z/q}{\leftthreetimes} (X\vee Y)^{[q]}$ induces an equivalence 
\[
\Omega^\infty\Sigma^\infty(\Sigma|X^{[q-1]} \wedge Y|) \to \underset{\underset n{\longrightarrow}}{\lim}\; \Omega^m hofibre(\tilde D_q(X\vee\Sigma^m Y)\to \tilde D_q(X)) = (D_1\tilde D_q)_X(Y) 
\]
\end{enumerate}
\end{proposition}

\begin{proof}  1) and 2) appear in \cite{ccgh}; the easiest way to verify them is to first compute $(D_1 F_q)_X(Y)$, which is straightforward, and then observe that inclusion the term $(E\Bbb Z/q\underset{\Bbb Z/q}{\leftthreetimes} (_-))$ has the effect of ``dividing by $q$''\ (in Goodwillie's words) via the fold map.  Finally 3) follows from 1) and 2), as the inclusion $X^{[q-1]} \wedge Y\to (X\vee Y)^{[q]}$ induces a map $\Omega^\infty\Sigma^\infty(\Sigma(|X^{[q-1]} \wedge Y|)) \to (D_1F_q)_X(Y)$ which agrees up to homotopy with the infinite loop extension of the inclusion of $X^{[q-1]} \wedge Y$ into the last term in the wedge $\overset q{\underset{i=1}{\vee}} X^{[i-1]} \wedge Y\wedge X^{[q-i]}$.
\end{proof}

Recall that for a ring $R$ and $r\in R$ the elementary matrix $e_{ij}(r)$ is the matrix $id + \overline e_{ij}(r)$ where $\overline
e_{ij}(r)_{k,\ell} = r$ if $(k,\ell) = (i,j), 0$ otherwise.  One should not try to push the analogy between $H^n_q(|JX|)$ and the group $GL_q(\Bbb Z[JX])$ too far, especially for finite $n$.  However one can construct elements of $H^n_q(|JX|)$ which behave enough like
elementary matrices to be useful.  We call these \underbar{elementary expansions}, as they correspond to the elementary expansions of classical Whitehead simple homotopy theory.

\begin{definition}\label{def:expand} Let $X$ be a connected  simplicial set, and $\imath : |X| \to |JX|$ the standard inclusion.  For $x\in |X|, e_{ij}(\imath(x)) \in |H^n_q(|JX|)|$ is given on $(S^n)_\ell \subset \overset q{\underset{k=1}{\vee}} (S^n)_k$ by 
\begin{align*}
&\underline{\ell\ne i}\qquad (S^n)_\ell \overset{ inc}{\longrightarrow} (S^n \wedge |JX|_+)_\ell\\ 
&\underline{\ell = i} \qquad (S^n)_i = S^n \overset{\text{pinch}}{\longrightarrow} S^n \vee S^n \overset{id\; \vee f}{\longrightarrow} S^n \vee (S^n\wedge |JX|)\overset{inc}{\longrightarrow} (S^n\wedge |JX|_+)_i \vee(S^n\wedge |JX|_+)_j
\end{align*}
\end{definition}

\noindent where (as before) we have identified $H^n_q(|JX|)$ with $\overline M^n_q(|JX|_+)$.  The sequence for $\ell = i$ is exactly as in (\ref{eqn:elem1}) with $f(s) = [s,\imath(x)] \in S^n \wedge |JX|$; the only difference is the indexing of the last term.  $e_{ij} (-\imath(x))$ is defined the same way, but with $id\vee f$ replaced by $id\vee (-f)$ where $-f$ is $f$ composed with a fixed choice of $S^n\overset{(-1)}{\longrightarrow} S^n$ representing loop inverse. The \underbar{reduced elementary expansion} $\overline e_{ij}(\imath(x))$ is given by
\begin{align*}
&\underline{\ell \ne i}\qquad\qquad\qquad (S^n)_\ell \longrightarrow * \\
&\underline{\ell = i}\qquad(S^n)_i \overset f{\rightarrow} (S^n \wedge |JX|) \overset {inc}{\longrightarrow} (S^n \wedge |JX|_+)_j.
\end{align*} 

Similarly one can define $\overline e_{ij}(-\imath(x))$.

\begin{remark}\label{rem:pm} When $i= j$, one could define $e_{ii}(\pm\imath(x))$\;\; (loop inverse) in a similar fashion.  Note that the definition of $e_{ij} (\pm(x))$ depends on a choice of parameters: choice of pinch map, choice of $j : S^n \wedge |JX| \to S^n \wedge |JX|_+$, and choice
of $S^n\overset{(-1)}{\longrightarrow} S^n$ representing $-1$. These, however, can be fixed so as to be compatible under suspension in the $n$ coordinate and depending in a continuous and natural way on $x\in |X|$ as well as $X$.  We assume this has been done.  In this way, all of the manipulations we will do with these elements will be natural in $X$ and $x\in |X|$.
\end{remark}

In a similar vein we will sometimes want to know that two maps depending on $x\in |X|$ (or diagrams depending on $X,Y,\dots$) are homotopic by a natural homotopy which depends continuously on $x\in |X|$ (resp\@. naturally homotopy-commutative by a homotopy which depends continuously on the spaces $X,Y,\dots$).   When this can be done, we will say the two maps are \underbar{canonically} homotopic
(or that the diagram is \underbar{canonically} $h$-commutative).

\begin{proposition}\label{prop:inv} Suppose $f = e_{i_1 j_1}(\imath(x_1))\cdot\dots\cdot e_{i_n j_n} (\imath(x_n))$ for $x_i \in |X|$. Then there is a canonical homotopy $f\cdot f^{-1} \simeq *$, where $f^{-1} = e_{i_n j_n} (-\imath(x_n))\cdot\dots\cdot e_{i_1 j_1} (-\imath(x_1))$.
\end{proposition}

\begin{proof} There is certainly a homotopy.  It can be made canonical by concentrating the homotopy in the spherical coordinates. This involves choosing a homotopy between 
\begin{align*}
&S^n \overset{\text{pinch}}{\longrightarrow} S^n \vee S^n
\xrightarrow{\text{pinch }\vee\text{ id}} S^n \vee S^n \vee S^n \\
\qquad\qquad\text{and}\\
&S^n \xrightarrow{\text{pinch}} S^n \vee S^n 
\xrightarrow{\text{id }\vee\text{ pinch}} S^n \vee S^n \vee S^n
\end{align*}
as well as a homotopy between $S^n \overset{\text{pinch}}{\longrightarrow} S^n \vee S^n  \overset{\text{id }\vee\text{ (-1)}} {\longrightarrow} S^n \vee S^n \overset{\text{fold}}{\longrightarrow} S^n$  and the trivial map $S^n \longrightarrow *$
\end{proof}

We are \underbar{not} making any claims that this homotopy is unique, even up to homotopy.  We will also need

\begin{proposition}\label{prop:canon1} For $x_1,\dots,x_{q-1} \in |X|,  y\in |Y|$, there is a canonical homotopy between 
\begin{gather*}
e_{12}(-\imath(x_1))\cdot e_{23} (-\imath(x_2))\cdot\dots\cdot
e_{q-1q}(-\imath(x_{q-1}))\overline e_{q1}(\imath(y)) \text{ and }\\
\overline e_{11}((\overset{q-1}{\underset{i=1}{\prod}}
-\imath(x_i))\imath(y)) + \overline e_{21}((\overset{q-1}{\underset{i=2}{\prod}}
-\imath(x_i))\imath(y)) + \dots + \overline e_{q1}(\imath(y))
\end{gather*}
where ``+''\ denotes loop sum.
\end{proposition}

\begin{proof} On the level of matrices this is clear; the product here is taking place in $|J(X\vee Y)|$.  Properly speaking, we should write $\overset{q-1}{\underset{i=j}{\prod}} -\imath(x_i)$ as $(-1)^{q-1-j} \overset{q-1}{\underset{i=j}{\prod}} \imath(x_i)$, given that $|J(X\vee Y)|$ is a monoid without any strict inverses.  To realize that the obvious homotopy is canonical, we note that it involves\newline  i) reparamerization to pass between the sequence of pinch maps used to evaluate the compositions and \newline ii)\; reparametrization to reposition the iterated power of $(-1)$ appearing in the expression $(-1)^{q-1-j} \overset{q-1}{\underset{i=j}{\prod}} \imath(x_i)$.  Both of these can be done in a natural and continuous way with respect to the parameters  $x_1,\dots,x_{q-1},y$ involved.
\end{proof}

The next result relates the representations $\rho_q$ of (\ref{eqn:elem1}) to products of elementary expansions.  This will be needed for the computation of the trace on $\tilde\rho_q$ given in \S3.3.  We define representations $\overline\rho_q^1, \overline\rho^2_q$ as follows:
\begin{align}\label{eqn:row12}
&\overline\rho^1_q(x_1,\dots,x_{q-1}) = \rho_q(x_1,\dots,x_{q-1},*)\quad x_i\in X\\
&\overline\rho^2_q(y) = p_2\rho_q(*,*,\dots,*,y) \quad  y\in Y
\end{align}

where $p_2 : H^n_q(|J(X\vee Y)|) \to M^n_q(|\overline F_1(X,Y)|)$ is as in Lemma \ref{lemma:p1p2}, with $X$ and $Y$ connected.

\begin{proposition}\label{prop:canon2} As continuous maps $\overline \rho^1_q$ and $\overline\rho^2_q$ are canonically homotopic to the following products of elementary expansions: 
\begin{gather*}
\overline\rho^1_q(x_1,\dots,x_{q-1}) \simeq
e_{(q-1)q}(\imath(x_{q-1})) e_{(q-2)(q-1)}(\imath(x_{q-2}))\cdot\dots\cdot e_{12}(\imath(x_1));\\
\overline\rho^2_q(y) \simeq \overline e_{q1}(y).
\end{gather*}
\end{proposition}

\begin{proof} This again only involves a reparametrization in the spherical coordinate independent of $X$ and $Y$, in the case of $\overline\rho^1_q$.  In the case of $\overline\rho^2_q$ we needn't do anything, as the projection map $p_2$ kills the identity maps along
the diagonal and we are left with a single non-zero entry. 
\end{proof}

\begin{remark} The above canonical homotopies arise from Steinberg identities, which hold in $H^n_k(|GX|)$ up to canonical homotopies.  Most types of identities among elementary expansions which hold up to homotopy do not hold up to canonical homotopy.  For example, it is not true that the entire representation $\rho_q$ is \underbar{canonically} homotopic to a product of elementary expansions.  This type of issue typically arises whenever one tries to analyze such cyclic representations in terms of products of elementary expansions.
\end{remark}

We have stated the above results using elementary expansions with entries in $\imath(|X|) \subset J|X|$, which is all we will need for section 3.  However the above constructions apply to the more general case where one allows arbitrary entries in $J|X|$ (or even $|GX|$ when $|X|$ is not a suspension).  Thus for $y\in J|X| \cong |JX|$, one defines $e_{ij}(y) \in |H^q_n(|JX|)|$ exactly as in Definition \ref{def:expand}, where $f : S^n \to S^n \wedge |JX|$ is the map $f(s) = [s,y] \in S^n \wedge |JX|$.  Similarly for the reduced elementary expansion $\overline e_{ij}(y)$.  Remark \ref{rem:pm} and Propositions \ref{prop:inv}, \ref{prop:canon1}, and \ref{prop:canon2} apply in this more general context.

\begin{proposition}\label{prop:canon3} For $a_1,\dots,a_{q-1} \in |JX|,\; b \in |\overline F_1(X,Y)|$ there is a canonical homotopy
between 
\begin{align*}
&e_{12}(-a_1) \cdot e_{23}(-a_2)\cdot\dots\cdot e_{q-1q}(-a_{q-1})\cdot\overline e_{q1}(b)\;\;\text{ and } \\
&\overline e_{11}((-1)^{q-1}(\overset{q-1}{\underset{i=1}{\prod}} a_i)b) + \overline e_{21}((-1)^{q-2}(\overset{q-1}{\underset{i=2}{\prod}} a_i)b) + \dots +  \overline e_{q-11} ((-1)a_{q-1}b) + \overline e_{q1} (b).
\end{align*}
\end{proposition}

The representations $\rho_q$ also extend in a natural way to yield a continuous map $\rho_q : |JX|^q \to|H^n_q(|JX|)|$, which on a
$q$-tuple $(a_1,\dots,a_q)\in |JX|^q$ is given exactly as in (\ref{eqn:elem1}), where $f_i$ is now the map $f_i(s) = [s,a_i] \in S^n \wedge |JX|$. Proposition \ref{prop:extend} applies with $|JX|$ in place of $|X|$ for the domain of $\tilde\rho_q$; in fact it is easy to see that the map defined in that proposition factors by this extension.

\begin{proposition} For $a_1,\dots,a_{q-1} \in |JX|, b \in |\overline F_1(X,Y)|$, let 
\begin{gather*}
\overline\rho^1_q (a_1,\dots,a_{q-1}) = \rho_q(a_1,a_2,\dots,a_{q-1},*)\\
\overline\rho^2_q(b) = p_2\rho_q(*,*,\dots,*,b)
\end{gather*}
as in (\ref{eqn:row12}).  Then as continuous maps $\overline\rho^1_q$ and $\overline\rho^2_q$ are canonically homotopic to the following product of elementary expansions: 
\begin{gather*}
\overline\rho^1_q(a_1,a_2,\dots,a_{q-1})\simeq
e_{(q-1)q}(a_{q-1})e_{(q-2)(q-1)}(a_{q-2})\cdot\dots\cdot e_{12}(a_1) \\
\overline\rho^2_q(b)\simeq \overline e_{q1}(b)\;.
\end{gather*}
\end{proposition}
The proofs of these two propositions is exactly as before.
\vskip.2in

Observe that when $a_1 = a_2 = \dots = a_{q-1} = *$, the map  $\overline\rho^1_q(a_1,\dots,a_{q-1})=\overline\rho^1_q(*,\dots,*)$ is not the standard inclusion $\overset{q}{\vee}(S^n) \hookrightarrow  \overset{q}{\vee} (S^n\wedge|JX|_+)$, only homotopic to it. This homotopy, which we will need later on, is a wedge of homotopies between
\begin{align}
S^n \overset{\text{pinch}}{\longrightarrow} S^n\vee S^n
&\overset{\text{id }\vee *}{\longrightarrow} S^n\vee S^n
\overset{\text{fold}}{\longrightarrow} S^n\nonumber \\
&\text{and}\label{eqn:1.3.16}\\
S^n &\overset{\text{id}}{\longrightarrow} S^n\nonumber
\end{align}
This wedge  produces a path between the basepoint of $H^n_q(|J(X)|)$ and $\overline\rho^1_q(*,\dots,*)$.
\vskip.2in
As a final remark, we note that in the above propositions involving minus signs, we are not requiring any type of coherence conditions to apply for this minus sign with respect to composition product (which in the limiting case $n\to\infty$ will involve the product structure on the generalized ring $\Omega^\infty\Sigma^\infty(|GX|_+)$).  We only require that certain homotopies can be made canonical. The restrictions on the ``ring'' under consideration that must be made in order for such a coherent $(-1)$ to exist are substantial, as shown by Schw\"anzl and
Vogt in \cite{sv}.
\newpage

%%%%%%%%%%%%%%%%%%%%%%%%%%%%%%%%%%%%%%%%%%%%%%%%%%%%%%%%%
%%%%%%%%%%%%%%%%%%%%%%%%%%%%%%%%%%%%%%%%%%%%%%%%%%%%%%%%%

\section{The trace}

%%%%%%%%%%%%%%%%%%%%%%%%%%%%%%%%%%%%%%%%%%%%%%%%%%%%%%%%%

\subsection{Manipulation in the stable range}

We follow closely the argument of Waldhausen [W2] in proving

\begin{theorem}\label{thm:2.1.1} Let $X$ and $Y$ be pointed simplicial sets, with $X$ connected and $Y$ $m$-connected.  Then the
two spaces 
\[
NH^n_k(|J(X\vee Y)|) \qquad \text{ and }\qquad N^{cy}(H^n_k(|JX|),M^n_k(\Sigma|\overline F_1(X,Y)|))
\]
 are $q$-equivalent, where $q =\min(n-2, 2m+1)$ and $n\ge 1$.
\end{theorem}

\begin{proof} The notation is that of \S2.1.  Here the monoid structure on $H^n_k(|J(X\vee Y)|)$ and $H^n_k(|JX|)$ is the usual one,
while the partial monoid structure on the $H^n_k(|JX|)$-bimonoid $M^n_k(|\overline F_1(X,Y)|)$ is the trivial one.  The equivalence follows as
in [Theorem 3.1, W2] by the construction of five maps, each of which is suitably connected.
\vskip .2in

\underbar{The 1st map} $H^n_k(|J(X\vee Y)|)$ admits a partial monoid structure where two elements are composable iff at most one of them lies outside the submonoid $H^n_k(|JX|)$.  The nerve of this partial monoid is by definition the generalized wedge 
\[
\{[p] \mapsto \overset p{\vee} (H^n_k(|J(X\vee Y)|), H^n_k(|JX|))\}.
\]
As $Y$ is $m$-connected, the inclusion $H^n_k(|JX|) \to H^n_k(|J(X\vee Y)|)$ is also $m$-connected.  It follows [Lemma 2.2.1, W2] that the inclusion
\[
\{[p] \mapsto \overset p{\vee} (H^n_k(|J(X\vee Y)|), H^n_k(|JX|))\} \hookrightarrow NH^n_k(|J(X\vee Y)|)
\]
is $(2m+1)$-connected.
\vskip .2in

\underbar{The 2nd map} The inclusion $F_1(X,Y) \hookrightarrow J(X,Y)$ is $(2m+1)$-connected, hence induces a $(2m+1)$-connected map $\overline M^n_k(|F_1(X,Y)|_+) \to \overline M^n_k(|J(X\vee Y)|_+)$ of $H^n_k(|JX|)$-bimonoids.  This in turn induces an inclusion of generalized
wedges 
\begin{align*}
\{[p] &\mapsto \overset p{\vee} (\overline M^n_k(|F_1(X,Y)|_+), H^n_k(|JX|))\}\\
\hookrightarrow\{[p] &\mapsto \overset p{\vee} (\overline M^n_k(J(X\vee Y)|_+), H^n_k(|JX|))\}
\end{align*}

which is $(2m+1)$-connected in each degree by the gluing Lemma [Lemma 2.1.2, W2] and induction on $p$.  It follows that the inclusion of simplicial objects is also $(2m+1)$-connected.
\vskip .2in

\underbar{The 3rd map} We consider the restriction to the path components corresponding to $H^n_k(|J(X\vee Y)|)$ of the inclusion 
\begin{align*}
M^n_k(|F_1(X,Y)|_+) &= Map( \overset k{\vee} S^n, \overset k{\vee} S^n \wedge |F_1(X,Y)|_+)\\ 
&\hookrightarrow Map ( \overset k{\vee} S^n,\overset k{\prod} S^n \wedge |F_1(X,Y)|_+)\\ 
&\simeq \overset k{\prod}\,\overset k{\prod}\, \Omega^n\Sigma^n (|F_1(X,Y)|_+).
\end{align*}
This is an $(n-1)$-equivalence.  [Lemma 1, W2] yields an $(n-2)$-equivalence 
\[
\Omega^n\Sigma^n(|F_1(X,Y)|_+) \simeq
\Omega^n \Sigma^n (|JX_+\,\vee\, \overline F_1(X,Y)|) \to \Omega^n\Sigma^n (|JX|_+) \times \Omega^n \Sigma^n(|\ov{F}_1(X,Y)|).
\]
The gluing lemma now applies to show that the map on nerves of partial monoids defined in Lemma \ref{lemma:p1p2}
\begin{align*}
\{[p] &\mapsto \overset p{\vee}(\overline M^n_k (|F_1(X,Y)|_+), H^n_k(|JX|))\}\\
\to \{[p] &\mapsto \overset p{\vee}(H^n_k(|JX|) \ltimes M^n_k (|\ov{F}_1(X,Y)|), H^n_k(|JX|))\}
\end{align*}
is $(n-2)$-connected.
\vskip .2in

\underbar{The 4th map} Taking the trivial monoid structure on $M^n_k((|\ov{F}_1(X,Y)|)$ and forming its nerve, [Lemma 2.3, W2] provides an equivalence 
\begin{align*}
\text{diag} (&N^{cy} (H^n_k(|JX|),\Sigma . M^n_k (|\ov{F}_1(X,Y)|)))\\
\overset{u}{\underset{\simeq}{\longrightarrow}} &N(H^n_k(|JX|)\ltimes M^n_k(|\ov{F}_1(X,Y)|))\\
= &\{[p] \mapsto \overset p{\vee} (H^n_k(|JX|) \ltimes M^n_k (|\ov{F}_1(X,Y)|), H^n_k(|JX|))\}.
\end{align*}

Here $\Sigma . A$ denotes the simplicial space $\{[p] \mapsto \overset p{\vee} (A,*)\}$ which arises on taking the nerve of a trivial partial monoid.
\vskip .2in

\underbar{The 5th map} Partial geometric realization produces a map from $\Sigma. M^n_k(|\ov{F}_1(X,Y)|)$ to $ S^1 \wedge M^n_k(|\ov{F}_1(X,Y)|)$.  The pairing map
\[
S^1 \wedge M^n_k(|\ov{F}_1(X,Y)|) \to M^n_k(S^1 \wedge |\ov{F}_1(X,Y)|)
\]
 together with partial geometric realization produces a map 
\[
N^{cy}(H^n_k(|JX|), \Sigma . M^n_k (|\ov{F}_1(X,Y)|))  \to N^{cy}(H^n_k(|JX|), M^n_k(S^1 \wedge |\ov{F}_1(X,Y)|)) \; .
\]
By the realization lemma, this map is $(2m+1)$-connected.
\vskip .2in

These 5 maps taken together yield the required sequence connecting the spaces $N(H^n_k(|J(X\vee Y)|))$ and $N^{cy} (H^n_k(|JX|), M^n_k (\Sigma |\ov{F}_1(X,Y)|))$.  Each of the maps is {$\min (n-2, 2m+1)$}-connected and the theorem follows.
\end{proof}

The maps constructed in the above theorem are compatible with suspension in the $n$-coordinate as well as pairing under block sum, by which we will always mean the wedge-sum of section 2.1 for the appropriate monoid in question.  Taking the limit as $n$ goes to $\infty$ yields a sequence of maps connecting 
\[
\underset{k\ge 0}{\coprod} N(H_k(|J(X\vee Y)|)) 
\qquad\text{and}\qquad  
\underset{k\ge 0}{\coprod} N^{cy} (H_k(|JX|),  M_k (\Sigma |\ov{F}_1(X,Y)|));
\]
each of these maps preserves block-sum and is $(2m-1)$-connected for $(m-1)$-connected $Y$.  We thus get a sequence of maps between their group completions which is also $(2m-1)$-connected.  
\vskip.2in

Denote $\Omega B(\underset{k\ge 0}{\coprod} N^{cy}(H_k(|JX|), M_k(\Sigma |\ov{F}_1(X,Y)|))$ by $C(X,Y)$, and $C(X,_-)$ by $C_X(_-)$;
$C_X(_-)$ is a homotopy functor on the category $T(*)$.

\begin{lemma} (compare [Lemma 4.2, W2])\label{lemma:4.2} There is an equivalence of 1st differentials 
\begin{align*}
(D_1 A\Sigma)_X(Y) &= \underset {n}{\underset{\longrightarrow}{\lim}}\;
\Omega^n(hofibre(A(\Sigma(X\vee (S^n \wedge Y))) \to A(\Sigma X))) \\
\simeq (D_1 C_X)_*(Y) &= \underset {n}{\underset{\longrightarrow}{\lim}}\;
\Omega^n(hofibre(C(X,S^n \wedge Y)\to C(X,*)))\;.
\end{align*}
\end{lemma}

\begin{proof} This is an immediate consequence of the above theorem; for each $n$, we have an equivalence $A(\Sigma X)\simeq C(X,*)$ as well as a $(2n-1)$-equivalence between  $A(\Sigma(X\vee (S^n \wedge Y)))$ and $C(X,S^n \wedge Y)$.  This gives a $(2n-1)$-equivalence
between $hofibre (A(\Sigma(X\vee (S^n \wedge Y))) \to A(\Sigma X))$ and $hofibre (C(X,S^n \wedge Y_+) \to C(X,*))$ which in the above limit
yields a weak equivalence.
\end{proof}
\vskip.5in

%%%%%%%%%%%%%%%%%%%%%%%%%%%%%%%%%%%%%%%%%%%%%%%%%%%%%%%%%

\subsection{The Generalized Waldhausen Trace}

In this section we construct a trace map, generalizing the construction of Waldhausen in \cite{w2}.  The techniques are
essentially those of [\S4, W2].
\vskip.2in

Let $F,F'$ be basepointed spaces, with $F$ $(m-1)$-connected.

\begin{lemma}\label{lemma:connected} For all integers $k,m,m\ge0$ the map of spaces 
\begin{gather*}
Map(\overset k{\vee} S^n, S^{n+m} \wedge F')\wedge Map (S^{n+m}, S^{n+m} \wedge F) \\
\overset\lambda{\longrightarrow} Map (\overset k{\vee} S^n, S^{n+m} \wedge F' \wedge F)\;,
\end{gather*}
 
given by 
\[
\lambda(f\wedge g) : \overset k{\vee} S^n \overset f{\rightarrow} S^{n+m}\wedge F' \overset{g\wedge id}{\longrightarrow} S^{n+m} \wedge F\wedge F' \xrightarrow[\cong]{id\wedge\text{switch}}
S^{n+m} \wedge F'\wedge F\;, 
\]
is $(3m-1)$-connected.
\end{lemma}

\begin{proof} This is a slight generalization of [Lemma 4.3, W2], and the proof is the same.  Namely, there is a commutative diagram 
\vskip.1in
\[
\diagram
Map(S^n,S^{n+m} \wedge F')^k \wedge Map (S^{n+m}, S^{n+m} \wedge F) \drto & \\
& Map(S^n, S^{n+m} \wedge F' \wedge F)^k \\
\overset k{\vee}(S^0 \wedge S^m\wedge F') \wedge Map (S^{n+m}, S^{n+m} \wedge F) \uuto & \\
\overset k{\vee}(S^0\wedge S^m\wedge F') \wedge Map(S^m, S^m\wedge F) \uto & 
\overset k{\vee} S^0\wedge(S^m\wedge F'\wedge F) \xto[-2,0] \\
\overset k{\vee}(S^0\wedge S^m\wedge F') \wedge F \uto \urto_{\cong}|<{\rotate\tip} & \\
\enddiagram
\] 
\vskip.1in
where the top horizontal map corresponds to the map given above, the right vertical map is $(4m-1)$-connected and each of the left vertical maps is $(3m-1)$-connected.
\end{proof}

For connected $Y'$, we have a homeomorphism of $|JX|$-bimonoids
\[
|\overline F_1 (X,Y')|\cong |JX|_+ \wedge |Y'| \wedge |JX|_+.
\]
If $Y' = S^m \wedge Y_+$ then the above lemma applies with $F = \overset k{\vee}|Y'|\wedge |JX|_+$ and $F' = |JX|_+$.  This yields a sequence of maps
\begin{gather}
H^n_k(|JX|)^p \times Map (\overset k{\vee} S^n, \overset k{\vee} S^{n+m} \wedge |\overline F_1(X,Y')|)\nonumber\\
\cong \updownarrow\varphi^p_{m,n,k}\nonumber\\
H^n_k(|JX|)^p \times Map (\overset k{\vee} S^n, \overset k{\vee} S^{n+m} \wedge |JX|_+ \wedge |Y'| \wedge |JX|_+)\label{eqn:2.2.2}\\
\uparrow f^p_{m,n,k}\nonumber\\
H^n_k(|JX|)^p \times (Map(\overset k{\vee} S^n,S^{n+m} \wedge |JX|_+) \wedge Map (S^{n+m},\,\overset k{\vee} S^{n+m} \wedge |Y'|\wedge |JX|_+))\nonumber\\
\downarrow g^p_{m,n,k}\nonumber\\
Map(S^{n+m}, \, S^{n+3m} \wedge |JX|_+ \wedge |Y_+|).\nonumber
\end{gather}

Here $f^p_{m,n,k}$ is induced by the pairing in Lemma \ref{lemma:connected}, and is $(3m-1)$-connected.  The map $g^p_{m,n,k}$ associates to the $(p+2)$-tuple $(\alpha_1,\dots,\alpha_p;\beta_1\wedge \beta_2)$ the composition

\begin{align}\label{eqn:2.2.3}
S^{n+m} \overset{\beta_2}{\longrightarrow} &\overset k{\vee} S^{n+m} \wedge |Y'|\wedge |JX|_+\nonumber\\
\overset\cong{\longleftrightarrow} &\overset k{\vee} (S^n\wedge |JX|_+)\wedge (S^m\wedge |Y'|)\nonumber\\
\xrightarrow{(\alpha_1\alpha_2\cdot\dots\cdot\alpha_p)\wedge id} &\overset k{\vee} (S^n\wedge |JX|_+)\wedge (S^m \wedge |Y'|)\nonumber\\ 
\overset\cong{\longleftrightarrow} (&\overset k{\vee} S^n) \wedge (|JX|_+\wedge S^m \wedge |Y'|) \\ 
\xrightarrow{\beta_1\wedge id}(&S^{n+m}\wedge |JX|_+)\wedge (|JX|_+ \wedge S^m \wedge |Y'|)\nonumber\\
\overset\cong{\longleftrightarrow} &S^{n+m}\wedge (|JX|_+ \wedge |JX|_+)\wedge (S^m \wedge |Y'|)\nonumber\\ 
\xrightarrow{id\wedge\mu\wedge id} &S^{n+m} \wedge |JX|_+ \wedge S^m \wedge |Y'|\nonumber\\
\overset\cong{\longleftrightarrow} &S^{n+3m} \wedge |JX|_+ \wedge |Y_+|.\nonumber
\end{align}

The equivalence of trivial partial monoids
\[
Map (\overset k{\vee} S^n, \overset k{\vee} S^{n+m} \wedge |JX|_+ \wedge |Y'| \wedge |JX|_+) \cong  Map (\overset k{\vee} S^n, \overset k{\vee} S^{n+m} \wedge |\overline F_1(X,Y')|)
\]
which induces the equivalence $\varphi^p_{m,n,k}$, commutes with the natural left and right actions of $H^n_k(|JX|)$, and therefore is an $H^n_k(|JX|)$-bimonoid equivalence. 

Starting with maps $f\in H^n_k(|JX|), g\in Map(\overset k{\vee} S^n, S^{n+m} \wedge |JX|_+), h \in Map(S^{n+m}, \overset k{\vee} S^{n+m} \wedge |Y'|\wedge |JX|_+)$ the pairings
\begin{gather}
(f,g) \mapsto f\cdot g,\nonumber \\ 
f\cdot g : \overset k{\vee} S^n
\overset\cong{\longleftrightarrow} \overset k{\vee} S^n \wedge S^0 \overset{id\wedge \iota_*}{\hookrightarrow} \overset k{\vee} (S^n \wedge |JX|_+) \overset f{\rightarrow} \overset k{\vee} (S^n \wedge |JX| _+)\label{eqn:2.2.4.a} \\
\overset\cong{\underset{\beta}{\longrightarrow}} (\overset k{\vee} S^n)\wedge  |JX|_+ \overset{g\wedge id}{\longrightarrow} (S^{n+m} \wedge |JX|_+)\wedge |JX|_+\overset\cong{\longleftrightarrow}\nonumber\\ 
S^{n+m} \wedge (|JX|_+\wedge |JX|_+)\overset{id\wedge\mu}{\longrightarrow} S^{n+m} \wedge |JX|_+\nonumber
\end{gather}

and

\begin{gather}
(h,f)\mapsto h\cdot f, \nonumber\\ 
h\cdot f : S^{n+m} \overset h{\rightarrow} \overset k{\vee} S^{n+m} \wedge |Y'| \wedge |JX|_+
\overset\cong{\underset{\alpha}\longrightarrow} S^m \wedge |Y'|\wedge (\overset k{\vee} S^n\wedge |JX|_+)\label{eqn:2.2.4.b} \\
\overset{id\wedge f}{\longrightarrow} S^m \wedge |Y'|\wedge (\overset k{\vee} S^n\wedge |JX|_+) \overset\cong{\underset{\alpha^{-1}}\longrightarrow} \overset k{\vee} S^{n+m} \wedge |Y'| \wedge |JX|_+ \nonumber
\end{gather}
 
induce a left $H^n_k(|JX|)$-monoid structure on $Map(\overset k{\vee} S^n, S^{n+m}\wedge |JX|_+)$ and a right $H^n_k(|JX|)$-monoid structure on $Map(S^{n+m}, \overset k{\vee} S^{n+m} \wedge |Y'|\wedge |JX|_+)$ for appropriate choices of homeomorphisms $\alpha,\beta$ and inclusion $\iota_*$. Precisely,
\begin{itemize}
\item $\iota_*:S^0\to |JX|_+$ is induced by the inclusion of the basepoint of $|JX|$;
\item $\beta$ is the homeomorphism $\overset k{\vee} (S^n \wedge |JX|_+)\cong (\overset k{\vee} S^n)\wedge  |JX|_+$ given as $\beta = \underset{j=1}{\overset k{\vee}} (\iota_j\wedge id)$, with $\iota_j$ denoting the inclusion of $S^n$ as the $j^{\text{th}}$ term in the wedge $\overset k{\vee} S^n$;
\item $\alpha$ is induced by fixed choice of homeomorphism $\overset k{\vee} S^{n+m}\cong S^m\wedge(\overset k{\vee} S^n)$, extended in the obvious way to include the $|Y'|$ and $|JX|_+$-factors.
\end{itemize}

Taken together, these actions define an $H^n_k(|JX|)$-bimonoid structure on the space \newline
$Map(\overset k{\vee} S^n,S^{n+m}\wedge |JX|_+) \wedge Map(S^{n+m}, \overset k{\vee} S^{n+m} \wedge |Y'|\wedge |JX|_+)$.  With this structure, the pairing map of Lemma \ref{lemma:connected} (for $F=\overset k{\vee} S^m \wedge |JX|_+$ and $F' = |JX|_+ \wedge |Y'|$) becomes an $H^n_k(|JX|)$-bimonoid map.
\vskip.2in

Although the maps $\alpha$ and $\beta$ above depend on $k$ and $n$, they clearly can be chosen so as to be compatible with any particular fixed choice of stabilization in the $k$-coordinate and suspension in the $n$-coordinate. Consequently, all of the above maps in (\ref{eqn:2.2.4.a}) and (\ref{eqn:2.2.4.b}) can be done compatibly with respect to both stabilization in the $k$-coordinate and suspension in the $n$-coordinate.  The equivalence $|\overline F_1 (X,S^m\wedge Y_+)\, |\overset\cong{\to} |JX|_+ \wedge S^m \wedge |Y_+| \wedge |JX|_+$ is compatible with suspension in the $m$-coordinate and so $\varphi^p_{m,n,k}$ in (\ref{eqn:2.2.2}) is compatible with suspension in both the $n$- and $m$-coordinates.  By much the same reasoning, the homeomorphisms appearing in (\ref{eqn:2.2.3}) involve a choice of natural equivalences, which can be chosen so as to be compatible with respect to suspension in these coordinates.  And with $\alpha_1, \beta_1$ and $\beta_2$ as in (\ref{eqn:2.2.3}) it follows directly from (\ref{eqn:2.2.3}) and (\ref{eqn:2.2.4.a}) that $g^1_{m,n,k}(\alpha_1 ; \beta_1 \wedge \beta_2) = g^0_{m,n,k}((\alpha_1\beta_1)\wedge \beta_2) = g^0_{m,n,k}(\beta_1\wedge(\beta_2\alpha_1))$.  Finally, by construction $\underset{k\ge 0}{\coprod} g^p_{m,n,k}$ maps wedge-sum to loop sum.  Putting this all together, we get

\begin{theorem}\label{thm:2.2.5} For each $n,m,k \ge 1, \{\varphi^p_{m,n,k}\}_{p\ge 0}, \{f^p_{m,n,k}\}_{p\ge 0}$ and
$\{g^p_{m,n,k}\}_{p\ge 0}$ induce well-defined maps of simplicial
spaces: 
\begin{gather*}
N^{cy}(H^n_k(|JX|) , M^n_k(S^m\wedge |\overline F_1 (X,Y')|)) \\
\cong\updownarrow\varphi^\cdot_{m,n,k} \\
N^{cy}(H^n_k(|JX|) , M^n_k (S^m\wedge |JX|_+ \wedge|Y'|\wedge |JX|_+)) \\
\uparrow f^\cdot_{m,n,k} \\ N^{cy}(H^n_k(|JX|) , Map (\overset k{\vee} S^n, S^{n+m} \wedge |JX|_+) \wedge Map (S^{n+m},\,\overset k{\vee} S^{n+m} \wedge |Y'|\wedge |JX|_+))\\
\downarrow g^\cdot_{m,n,k} \\
Map (S^{n+m}, S^{n+3m} \wedge |JX|_+ \wedge |Y_+|)), 
\end{gather*}

where the simplicial structure on the range of $(g^\cdot_{m,n,k})$ is trivial, $Y' = S^m \wedge Y_+$, and $f^\cdot_{m,n,k}$ is $(3m-1)$-connected.  These maps are compatible with suspension in the $m$ and $n$ coordinates, and stabilization in the $k$-coordinate.
They are also natural with respect to $X$ and $Y$ (where $X$ is connected).
\end{theorem}

\begin{proof} The only point that has not already been covered is the statement concerning the connectivity of $f^\cdot_{m,n,k}$.  But this follows by the realization lemma, as $f^p_{m,n,k}$ is $(3m-1)$-connected for each $p$.
\end{proof}

Now $g^\cdot$ takes wedge-sum to loop sum, as we have already noted, and thus factors via group completion with respect to wedge-sum.  So passing to the limit in $m$ yields a map $T$:
\begin{gather*}
(D_1C_X)_*(Y_+) \\ 
=\underset{m}{\underset{\longrightarrow}{\lim}}\;\Omega^m(hofibre (C(X, S^m \wedge Y_+) \to C(X,*))) \\
\overset T{\rightarrow}\; \Omega^\infty\Sigma^\infty(\Sigma(\underset{q\ge 1}{\vee}\; |X^{[q-1]} \wedge Y_+ |)).
\end{gather*}

Precomposing by the equivalence of Lemma \ref{lemma:4.2} we get 
\[ 
Tr_X(Y) :(D_1 A\Sigma)_X(Y_+) \to \Omega^\infty\Sigma^\infty(\Sigma(\underset{q\ge 1}{\vee} |X^{[q-1]} \wedge Y_+ |)).  
\]
(in the case $X = \{pt\}$ we recover the map constructed in \cite{w2}).  This map is natural in both $X$ and $Y$.  Taking the fibre with respect to the map $Y_+ \to \{pt\}$ yields (for basepointed $Y$) the (reduced) \underbar{generalized Waldhausen trace map}
\begin{equation}\label{eqn:gentrace}
\overline{Tr}_X(Y) : (D_1 A\Sigma)_X(Y) \to \Omega^\infty\Sigma^\infty(\Sigma(\underset{q\ge 1}{\vee} |X^{[q-1]} \wedge Y|)) 
\end{equation}
where on the right we have, for $q\ge 1$, composed with the (basepointed) projection $Y_+ \to Y$ . Finally, we can follow by projection to the $q$th factor $\Omega^\infty\Sigma^\infty(\Sigma|X^{[q-1]} \wedge Y|)$; this yields a map
\[
\overline{Tr}_X(Y)_q : (D_1  A\Sigma)_X(Y) \to \Omega^\infty\Sigma^\infty(\Sigma |X^{[q-1]} \wedge Y|).
\]
For connected $X$, $Tr_X(Y) \cong \underset{q\ge 1}{\prod} Tr_X(Y)_q$ .
\vskip.5in

%%%%%%%%%%%%%%%%%%%%%%%%%%%%%%%%%%%%%%%%%%%%%%%%%%%%%%%%%

\subsection{Computing the trace on $\ov\rho_q$}

By the results of section 1.3, there is a map
\[
\tilde\rho = \underset{q\ge 1}{\prod} \tilde\rho_q : \tilde D(X) =  \underset{q\ge 1}{\prod} \tilde D_q(X) \to \overline A(\Sigma X)
\] 
defined for any connected simplicial set $X$.  This map is natural in $X$,  and is induced by the representations $\rho_q : |X|^q \to |H^n_k(|JX|)|$.   The product that appears on the L.H.S. is the weak product; note, however, that as $X$ is connected, the weak product in this case is weakly equivalent to the strong product. Replacing $X$ by $X\vee Y$, we define $\rho_q(X,Y)$ as  the restriction of $\rho_q$ to  $|X|^{q-1} \times |Y| \subset |X\vee Y|^q$.  This  inclusion induces the inclusion $i_q(X,Y)$ of Proposition \ref{prop:rep1} after passing to smash products.  Let $\tilde i_q(X,Y) = \Omega^\infty\Sigma^\infty(\Sigma i_q(X,Y))$. Then the composition 
\begin{gather*}
\tilde\rho_q(X,Y) : \Omega^\infty\Sigma^\infty(\Sigma |X^{[q-1]} \wedge Y|) \xrightarrow{\tilde i_q(X,Y)}
 \Omega^\infty\Sigma^\infty (\Sigma |X\vee Y|^{[q]}) \to \\ 
\Omega^\infty\Sigma^\infty (\Sigma (E\Bbb Z/q\underset{\Bbb Z/q}{\leftthreetimes} |(X\vee Y)^{[q]}|)) = \tilde D_q(X\vee Y)
\overset{\tilde\rho_q}{\longrightarrow} \overline A(\Sigma(X\vee Y))
\end{gather*}
can alternatively be described as the precomposition of $\Omega^\infty\Sigma^\infty(\Sigma\rho_q(X,Y))$ with the stable section $s : \Omega^\infty\Sigma^\infty (\Sigma |X^{[q-1]} \wedge Y|) \to \Omega^\infty\Sigma^\infty (\Sigma |X|^{q-1} \times |Y|)$, followed by the map into $ (\overline A\Sigma)(X\vee Y)$.  Proposition \ref{prop:rep1} tells us that the map $\tilde i_q(X,Y)$ induces an equivalence $\Omega^\infty\Sigma^\infty(\Sigma  X^{[q-1]} \wedge Y|) \overset\simeq{\rightarrow} (D_1\tilde D_q)_X(Y)$, and Goodwillie's results tell us that $\tilde\rho$ is an equivalence for connected spaces iff $(D_1\tilde\rho)_X(Y)$ is an equivalence for all connected $X$.  They also tell us that $(D_1 A\Sigma)_X(Y)$ and $(D_1\tilde D)_X(Y)$ are the same for connected $X$. Thus the primary task is to show that  $\overline {Tr}_X(Y) \circ (D_1\tilde\rho)_X(Y)$ is an equivalence for all connected $X$.

\begin{theorem}  For $p\ne q, \overline {Tr}_X (Y)_p \circ \tilde\rho_q (X,Y) \simeq *$.  When $p = q$, $ \overline Tr_X (Y)_q \circ \tilde\rho_q (X,Y) \simeq (-1)^{q-1}$.  These homotopies are canonical in $X$ and $Y$, and hold for all connected $X$ and $q\ge 1$.  Thus $\overline Tr_X (Y) \circ (\underset{q\ge 1}{\prod} \tilde\rho_q (X,Y))$ is an equivalence for connected $X$, which implies $\overline {Tr}_X (Y) \circ (D_1\tilde\rho)_X(Y)$ is an equivalence for connected $X$.
\end{theorem}

\begin{proof} The last implication follows by Proposition \ref{prop:rep1}. Our main objective is the evaluation of the trace map $\overline {Tr}_X(Y)$ on $\tilde\rho_q(X,Y)$, which we will do in stages.  First, we determine what happens to the image of the representation $\rho_q(X,Y)$ under the maps constructed in Theorem \ref{thm:2.1.1}.  This will bring us into the cyclic bar construction.  The maps provided by (\ref{eqn:2.2.2}) -- (\ref{eqn:2.2.4.b}) will then determine the composition $\overline Tr_X(Y)$.
\vskip.2in

We will assume $Z' = X\vee Y$ where $X$ is connected and $Y$ is $m'$-connected.  $\rho_q$ (resp.\  its restriction $\rho_q (X,Y)$) is induced by a simplicial representation ${Z'}^q$ (resp.\ $ X^{q-1} \times Y$) $\to H^n_q(|JZ'|)$ which we also denote by $\rho_q$ . This map can be represented simplicially by a map of partial monoids: 
\[
\{[p] \mapsto \overset p{\vee} ({Z'}^q,*)\} \overset{\{\overset p{\vee} \rho_q\}}{\longrightarrow} NH^n_q (|JZ'|).
\]

We will construct five diagrams, one for each of the maps in the proof of Theorem \ref{thm:2.1.1}.
\vskip.2in

\underbar{The 1st diagram} The first map in Theorem \ref{thm:2.1.1} was induced by the $(2m'+1)$-connected inclusion of partial monoids: 
\[ 
\{[p] \mapsto \overset p{\vee} (H^n_q (|JZ'|), H^n_q (|JX|))\} \overset{\imath_1}{\longrightarrow} NH^n_q (|JZ'|)\;.  
\] 
The generalized wedge on the left contains the image of $\{\overset p{\vee}\rho_q\}$ and hence $\{\overset p{\vee} \rho_q(X,Y)\}$.  Thus $\{\overset p{\vee}\rho_q (X,Y)\} = \imath_1 \circ \overline\rho_{q,1}$, were $\overline\rho_{q,1}$ is a map of generalized wedges, induced in each degree by the representation $\rho_q(X,Y)$, and fits into the commutative diagram: 
$$
\diagram
\{[p] \mapsto  \overset p{\vee} (X^{q-1} \times Y,*)\} \rto^(.56){\overset p{\vee}\rho_q(X,Y)} \ddouble & NH^n_q (|JZ'|)  \\
\{[p] \mapsto  \overset p{\vee} (X^{q-1} \times Y,*)\}      \rto^(.44){\overline\rho_{q,1}} & \{[p] \mapsto  \overset p{\vee} (H^n_q(|JZ'|), H^n_q (|JX|))\} \uto_{\imath_1} \\
\enddiagram
$$
\vskip.2in

\underbar{The 2nd diagram} The second map in Theorem \ref{thm:2.1.1} is the $(2m'+1)$-connected map of generalized wedges induced by the $(2m'+1)$-connected inclusion 
\[
\overline M^n_q (|F_1(X,Y)|_+) \to \overline M^n_q (|JZ'|_+) \cong H^n_q (|JZ'|).
\]
As the image of $\rho_q$ is contained in $\overline M^n_q (|F_1(X,Y)|_+)$, we can further factor $\rho_q(X,Y)$ as $\imath_2 \circ \overline\rho_{q,2}$. $\overline\rho_{q,2}$ is defined exactly as $\overline\rho_{q,1}$ -- it is the (unique) map of generalized wedges induced by $\rho_q(X,Y)$ which makes the following diagram commute: 
$$
\diagram
\{[p] \mapsto  \overset p{\vee} (X^{q-1} \times Y,*)\} \rto^(.44){\overline\rho_{q,1}} \ddouble & \{[p] \mapsto \overset p{\vee}(\overline M^n_q (|JZ'|_+), H^n_q (|JX|))\} \\ 
\{[p] \mapsto  \overset p{\vee} (X^{q-1} \times Y,*)\} \rto^(.40){\overline\rho_{q,2}} & \{[p] \mapsto \overset p{\vee}(\overline M^n_q (|F_1(X,Y)|_+), H^n_q (|JX|))\} \uto_{\imath_2} \\
\enddiagram
$$
\vskip.2in

\underbar{The 3rd diagram}\;\; The $(n-2)$-connected map
\[
\overline M^n_q (|F_1(X,Y)|_+)\overset {p_1 \times p_2}{\longrightarrow} H^n_q (|JX|)) \times M^n_q (|\overline F_1(X,Y)|)
\]
 induces the third map in Theorem \ref{thm:2.1.1}, where the projections $p_1,p_2$ are induced by the projections of $F_1(X,Y)$ to $JX$ and $\overline F_1(X,Y)$ respectively.  Let $\overline\rho^i_q = p_i \circ \rho_q(X,Y)$ for $i = 1,2$.  Then we have a commuting square
$$
\diagram
\{[p] \mapsto  \overset p{\vee} (X^{q-1} \times Y,*)\} \rto^(.41){\overline\rho_{q,2}} \ddouble &  \{[p] \mapsto \overset p{\vee}(\overline M^n_q (|F_1(X,Y)|_+), H^n_q (|JX|))\} \dto  \\
\{[p] \mapsto  \overset p{\vee} (X^{q-1} \times Y,*)\} \rto^(.32){\overline\rho_{q,3}} & \{[p] \mapsto \overset p{\vee}(\overline M^n_q(|JX|_+) \ltimes M^n_q (|\overline F_1(X,Y)|), H^n_q(|JX|))\}
\enddiagram
$$  
where $\overline\rho_{q,3}$ is induced in each degree by the product $\overline\rho^1_q \times \overline\rho^2_q$.
\vskip.2in

\underbar{The 4th diagram} This is the first place where one encounters complications in computing the trace map on arbitrary representations.  From equation (\ref{eqn:1.1.2}) we can see the problem -- when $M$ is not a group but only grouplike there may be no simple way to choose $f^{-1}$ for $f \in M$, which one needs to do in order to formally invert the equivalence $u : diag(N^{cy}(M,NE)) \overset\simeq{\longrightarrow} N(M\ltimes E)$.  In our case by first
reducing the representation under consideration to $\rho_q(X,Y)$ we are able to circumvent this difficulty.  For by Proposition \ref{prop:canon2}, $\overline\rho^1_q$ and $\overline\rho^2_q$ are canonically homotopic to a product of elementary expansions: 
\begin{gather}
\overline\rho^1_q(x_1,\dots,x_{q-1})\cong e_{q-1q}(\imath(x_{q-1})) e_{q-2q-1}(\imath(x_{q-2}))\cdot \dots \cdot e_{12}(\imath(x_1))\label{eqn:2.3.5.1} \\ 
\overline\rho^2_q(y) \cong \overline e_{q1}(\imath(y)) \text{ (the \underbar{reduced} expansion with $(q,1)$ entry $\imath(y)$)}\label{eqn:2.3.5.2}
\end{gather}
where $\imath(X)$ denotes the image of $x \in |X|$ in $|JX|$ under the natural inclusion $X \to JX$, and similarly for $Y$ (for notational simplicity, we have used $|\overline\rho^1_q|$ and $|\overline\rho^2_q|$ and $|Y|$ respectively.  To recover $\overline\rho^1_q$ and $\overline\rho^2_q$ as above one applies Sing
$(_-)$ and precomposes with the map $A \to Sing(|A|)$).  The notation is explained in section 2.3.  For such a product of elementary expansions Proposition \ref{prop:inv} yields a canonical homotopy between $f^{-1}f, *$ and $ff^{-1}$ where $f^{-1} = e_{12}(-\imath(x_1)) e_{23}(-\imath(x_2))\cdot \dots \cdot e_{q-1q}(-\imath(x_{q-1}))$ for $f =\overline\rho^1_q(x_1,\dots,x_{q-1})$ as above.  We can define a map
\begin{align*}
|\overline\rho^1_{q,4}| : |X|^{q-1} \times |Y| \to &|N^{cy}_1 (H^n_q(|JX|), M^n_q(|\overline F_1 (X,Y)|))| \\
&= H^n_q(|JX|) \times M^n_q(|\overline F_1 (X,Y)|)|
\end{align*}

by $(x_1,\dots, x_{q-1},y) \mapsto (f,f^{-1} ef^{-1})$ where $f = \overline\rho^1_q(x_1,\dots, x_{q-1}),e = \overline e_{q1}(y)$ as given above in (\ref{eqn:2.3.5.1}), (\ref{eqn:2.3.5.2}).  Extending degreewise yields a map $\overline\rho_{q,4}$ and a canonically homotopy-commutative diagram
$$
\diagram
\{[p] \mapsto  \overset p{\vee} (X^{q-1} \times Y,*)\} \rto^(.32){\overline\rho_{q,3}} \ddouble & \{[p] \mapsto \overset p{\vee}(\overline M^n_q(|JX|_+) \ltimes M^n_q (|\overline F_1(X,Y)|), H^n_q(|JX|))\}  \\
\{[p] \mapsto  \overset p{\vee} (X^{q-1} \times Y,*)\} \rto^(.37){\overline\rho_{q,4}} & diag(N^{cy} (H^n_q(|JX|),\Sigma . M^n_q (|\overline F_1(X,Y)|))) \uto_{\simeq}
\enddiagram
$$ 
\vskip.1in
where the equivalence on the right hand side is the map $u$ defined in (\ref{eqn:1.1.2}), and the bottom map has been modified by the homotopy of
(\ref{eqn:1.3.16}) applied to the map in (\ref{eqn:2.3.5.1}) to make it basepoint-preserving. Recall that $\Sigma. A$ is shorthand notation for $\{[p] \mapsto
\overset p{\vee}(A,*)\}$.  The fact that the diagram is canonically homotopy-commutative is important.  Note also that $\overline\rho_{q,4}$  is given on one-simplices by $\overline\rho^1_{q,4}$.  
\vskip.2in

\underbar{The 5th diagram} In Theorem \ref{thm:2.1.1} the fifth map is induced by partial geometric realization
\[
r: \Sigma . M^n_q (|\overline F_1(X,Y)|)\to S^1 \wedge M^n_q (|\overline F_1(X,Y)|)
\] 
and the pairing 
\[
p : S^1 \wedge M^n_q (|\overline F_1(X,Y)|) \to M^n_q (S^1 \wedge |\overline F_1(X,Y)|) \overset\cong{\longrightarrow} M^n_q (|\overline F_1(X,S^1 \wedge Y)|).
\]
Let $M^n_q (|\overline F_1(X,\Sigma. Y)|)$ denote the simplicial object $\{[p] \mapsto M^n_q (|\overline F_1(X,\overset p{\vee} (Y,*))|)\}$ where the face and degeneracy maps are induced by those of $\Sigma . Y$.  There is an obvious map of simplicial objects
\[
\Sigma . M^n_q(|\overline F_1(X,Y)|) \hookrightarrow  M^n_q(|\overline F_1(X,\Sigma . Y)|)
\]
which in degree $p$ is given by the inclusion
\[
\overset p{\vee} M^n_q(|\overline F_1(X,Y)|) \hookrightarrow M^n_q(|\overline F_1(X,\overset p{\vee} Y)|).
\]
The partial realization map r sends $M^n_q(|\overline F_1(X,\Sigma . Y)|)$ to $M^n_q(|\overline F(X,S^1 \wedge Y)|)$ and the composition 
\[
\Sigma .M^n_q(|\overline F_1(X,Y)|) \to M^n_q(|\overline F_1(X,\Sigma . Y)|)
\overset r{\rightarrow} M^n_q(|\overline F_1(X,S^1 \wedge Y)|)
\] 
is equivalent to the previous composition of partial realization followed by the pairing $p$.  Note that the partial realization map above is $(n-2)$-connected by the same type of argument used in the construction of the third map in Theorem \ref{thm:2.1.1}.  Now the map
\[
\Sigma . M^n_q(|\overline F_1(X,Y)|) \overset\alpha{\rightarrow} M^n_q(|\overline F_1(X,\Sigma . Y)|)
\]
is an $H^n_q(|JX|)$-bimonoid map, and so induces a bisimplicial map: 
\[ 
N^{cy} (H^n_q(|JX|),\Sigma . M^n_q (|\overline F_1(X,Y)|))\overset\beta{\rightarrow} N^{cy}(H^n_q(|JX|), M^n_q (|\overline F_1(X,\Sigma . Y)|)).
\]

Let $N^{cy}_p (M,NE)$ denote the simplicial object 
\[
\{[k] \mapsto N^{cy}_{p,k} (M,NE) = M^p \times (NE)_k\}.
\] 
Then the representation
\begin{gather*}
\rho^1_{q,4} : X^{q-1} \times Y \to H^n_q(|JX|) \times  M^n_q(|\overline F_1 (X,Y)|) = N^{cy}_{1,1} (H^n_q(|JX|), \Sigma . M^n_q(|\overline F_1 (X,Y)|)) \\ 
\overset\beta{\rightarrow} N^{cy}_{1,1} (H^n_q(|JX|), M^n_q (|\overline F_1(X,\Sigma . Y)|))
\end{gather*}

extends uniquely to a map of simplicial objects: 
\[ 
\overline\rho_{q,5} : X^{q-1} \times \Sigma . Y \to N^{cy}_1 (H^n_q(|JX|), M^n_q(|\overline F_1(X,\Sigma . Y)|)).
\] 
It is not true that there is a map $\Sigma . (X^{q-1} \times Y) \to X^{q-1} \times \Sigma . Y$ of simplicial objects which makes the appropriate diagram commute (here $X^{q-1} \times \Sigma . Y$  is the simplicial object  $\{[p] \mapsto X^{q-1} \times (\overset p{\vee}(Y,*))\}$). However there is after passing to smash products.  Specifically, we may choose stable splittings $i_1,i_2$ for which $p_1 \circ i_1 \simeq p_2 \circ i_2 \simeq id$ by a homotopy functorial in $X$ and $Y$. These produce the square
$$ 
\diagram
\Omega^{\infty} \Sigma^{\infty} (\Sigma |X^{[q-1]} \wedge Y|) \rto<1ex>^{i_1} \ddto_{\cong}|<{\rotate\tip} & 
\Omega^{\infty} \Sigma^{\infty} (\Sigma(|X|^{q-1}\times Y|)) \lto<1ex>^{p_1} \dto^{\tilde\rho_{q,4}}  \\
& \Omega^{\infty} \Sigma^{\infty}(|diagN^{cy} (H^n_q(|JX|),\Sigma . M^n_q(|\overline F_1 (X,Y)|))|) \ddtor^{\tilde\beta} \\
\Omega^{\infty} \Sigma^{\infty} (|X|^{[q-1]} \wedge \Sigma |Y|) \rto<1ex>^{i_2} & 
\Omega^{\infty} \Sigma^{\infty} (|X^{q-1} \times \Sigma . Y|) \lto<1ex>^{p_2}\dto^{\tilde\rho_{q,5}} \\
& \Omega^{\infty} \Sigma^{\infty} (|N^{cy}(H^n_q(|JX|), M^n_q(|\overline F_1 (X,\Sigma . Y)|))|)
\enddiagram
$$ 
where $\tilde\rho_{q,j} = \Omega^\infty\Sigma^\infty |\overline\rho_{q,j}|$ for $j= 4,5$ and $\overline\beta$ is induced by $(\beta)$.  By the construction of $\tilde\rho_{q,4}$ and $\tilde\rho_{q,5}$ it is straightforward to see that the diagram is canonically homotopy-commutative. Note that the space appearing in the lower right-hand corner is $(n-2)$-equivalent to $\Omega^\infty\Sigma^\infty(|N^{cy} (H^n_q(|JX|),M^n_q(|\overline F_1(X, S^1 \wedge Y)|)|)$ . This is our fifth diagram.
\vskip.2in

Before evaluating the trace we make a useful simplification.  In order to be consistent with notation, we will assume $Y = \Sigma^{2m-1}Z_+$ and use $\Sigma^{2m} F$ to denote $\Sigma|\overline F_1(X,Y)|$. There is no loss of generality here, because computation of $\overline {Tr}_X(Y)$ involves passing through a direct limit in which $Y$ becomes more and more highly suspended.  Now we know that the partial
realization map 
\[
r : N^{cy}(H^n_q(|JX|), M^n_q(|\overline F_1(X,\Sigma . Y)|)) \to N^{cy} (H^n_q(|JX|), M^n_q(|\overline F_1(X, S^1\wedge Y)|))
\]
commutes with the simplicial structure in the first coordinate (given by the face and degeneracy maps of the cyclic bar construction), and that it maps the simplicial space
\[ 
\{[k] \mapsto N^{cy}_{p}(H^n_q(|JX|), M^n_q(|\overline F_1(X,\{\Sigma. Y)\}_k|))\}
\]
to the space $N^{cy}_p (H^n_q(|JX|), M^n_q(|\overline F_1(X, S^1 \wedge Y)|))$.  It follows by Theorem \ref{thm:2.2.5} that, upon restricting to the $q$th component of the trace map  $\overline {Tr}_X(Y)$, we have a canonically homotopy commutative diagram
\[
\spreaddiagramrows{-1.0pc}
\spreaddiagramcolumns{-1.0pc}
\def\objectstyle{\ssize}
\def\labelstyle{\sssize}
\diagram
& & X^{q-1}\times (\Sigma . Y) \ddto^{\rho_{q,5}}  \\
M^n_q(|\overline F_1(X,\Sigma . Y)|) \ddto_{r} \rto^(.37){\cong} &  
N^{cy}_0 (H^n_q(|JX|),M^n_q(|\overline F_1(X, \Sigma . Y)|)) \ddto_{r} & \\
& & N^{cy}_1(H^n_q(|JX|),M^n_q(|\overline F_1(X, \Sigma . Y)|)) 
\ulto_{\partial_0} \ddto^{r}\\
M^n_q(\overline F_1(X,S^1 \wedge Y)|) \rto^(.37){\cong} & 
N^{cy}_0(H^n_q(|JX|),M^n_q(\overline F_1(X,S^1 \wedge Y)|)) & \\
& & N^{cy}_1 (H^n_q(|JX|),M^n_q(|\overline F_1(X, S^1\wedge Y)|)) 
\ulto_{\partial_0}  \\
N^{cy}_0 (H^n_q(|JX|),M^n_{q,1} \wedge M^n_{q,2}) 
\uuto^{\varphi^0_{m,n,q}\circ f^0_{m,n,q}} 
\ddrto_{\overline\pi_p \circ g^0_{m,n,q}} & & \\
& & N^{cy}_1 (H^n_q(|JX|),M^n_{q,1} \wedge  M^n_{q,2}) 
\uuto_{\varphi^1_{m,n,q}\circ f^1_{m,n,q}} \ullto_{\partial_0}
\dlto_{\overline\pi_p \circ g^1_{m,n,q}}  \\
& \Omega^{n+m}\Sigma^{n+3m} (|X^{[p-1]} \wedge Z_+|) &
\enddiagram
\] 
where 
\begin{gather*}
M^n_{q,1} = Map (\overset q{\vee} S^n,S^{n+m}\wedge |JX|_+\wedge |Z_+|) \\
M^n_{q,2} = Map (S^{n+m}, \overset q{\vee} S^{n+2m}\wedge |JX|_+)\;, 
\end{gather*}

and $\overline\pi_p$ is the obvious reduced projection onto the $p$th component 
\[
\Omega^{n+m}\Sigma^{n+3m} (|X^{[p-1]} \wedge Z_+|).
\]
Our object now is to find a map $\rho_{q,6}$ defined on $X^{q-1}\times (\Sigma . Y)$ or its realization, whose range is $N^{cy}_0 (H^n_q(|JX|),M^n_{q,1} \wedge M^n_{q,2})$ such that  $\varphi^0_{m,n,q}\circ f^0_{m,n,q} \circ \rho_{q,6}$ is canonically homotopic to $r\circ \partial_0 \circ \overline\rho_{q,5}$ (of course, it would suffice to lift $r\circ \overline\rho_{q,5}$ directly without using $\partial_0$, and in fact such a lifting can be written down explicitly.  However, it is much simpler to do this after mapping first by $\partial_0$; this makes the computation of $\overline\pi_p \circ g^0_{m,n,q}$ easier as well).  Now $\overline\rho_{q,5}$ is the unique extension to $X^{q-1} \times (\Sigma . Y)$ of the representation $\overline\rho^1_{q,4}$ on 1-simplices $X^{q-1} \times Y$ given by 
\begin{gather*}
(x_1,\dots,x_{q-1},y) \mapsto (f,f^{-1}ef^{-1});\\
f =\overline\rho^1_q(x_1,\dots,x_{q-1}),\, e=\overline\rho^2_q(y)\;. 
\end{gather*}
These are in turn expressed as a product of elementary expansions by (\ref{eqn:2.3.5.1}), (\ref{eqn:2.3.5.2}).  Under $\partial_0$ this element maps to $(f^{-1}ef^{-1}\cdot f)$ which is canonically homotopic to $(f^{-1}e)$.  It follows that we can describe $r \circ \partial_0 \circ \overline\rho_{q,5}$ on the realization of $X^{q-1} \times (\Sigma . Y)$ as the map of spaces given by the
representation 
\begin{gather*}
\rho^1_{q,6} : |X^{q-1} \times (\Sigma . Y)| \to |M^n_q(|\overline F_1(X, \Sigma . Y)|)| \\ 
(x_1,\dots,x_{q-1},\tilde y) \mapsto (f^{-1}\tilde e);\quad f =\overline\rho^1_q(x_1,\dots,x_{q-1}), \tilde e = \overline e_{q-1}(\imath (\tilde y))
\end{gather*}
where $f$ is now viewed as a product of elementary expansions whose range is $|M^n_q(|\overline F_1(X, \Sigma . Y)|)|$. Note that $\tilde y$ denotes an element of $|\Sigma . Y|$.  Writing $f^{-1}$ as $e_{12}(-\imath(x_1)) e_{23}(-\imath(x_2))\cdot\dots\cdot e_{q-1q}(-\imath(x_{q-1}))$ and applying Proposition \ref{prop:canon1} yields a canonical homotopy between $\overline\rho^1_{q,6}$ and the
representation
\begin{gather*}
\overline\rho^2_{q,6} : |X^{q-1} \times (\Sigma . Y)| \to |M^n_q(|\overline F_1(X, \Sigma . Y)|)| \\
(x_1,\dots,x_{q-1},\tilde y) \mapsto \overline e_{11}(z_1) +  \overline e_{21}(z_2) + \dots + \overline e_{q1}(z_q) 
\end{gather*}
where $z_i = (\overset{q-1}{\underset{j=i}{\prod}} - \imath(x_i)) \imath(\tilde y) \in \pm (|\overline F_1(X, \Sigma .Y)|) = \pm \Sigma^{2m} F$ and ``-'' denotes the inverse under loop sum. We can write $\tilde y \in \Sigma|Y| \cong S^m \wedge |Z| \wedge S^m$ as $\tilde y = (s_1,z,s_2)$.  Define $\overline\rho^3_{q,6}$ by
\begin{gather}
\overline\rho^3_{q,6} : |X^{q-1} \times (S^m \wedge |Z|_+ \wedge S^m)| \to |M^n_{q,1} \wedge M^n_{q,2}|\nonumber\\
(x_1,\dots,x_{q-1},s_1,z,s_2) \mapsto (\overline e_{11}(z'_1) + \dots + \overline e_{q1}(z'_q)) \wedge \imath^1(s_2).\label{eqn:2.3.11}
\end{gather}

Here $M^q_{n,i}$ is as in the last diagram.  The map $\imath^1(s_2)$ is represented by the composition
\begin{gather*}
S^{n+m} \to S^{n+m} \wedge S^m \cong S^{n+2m}\wedge S^0 \overset{\text{inc}}{\hookrightarrow} (S^{n+2m} \wedge |JX|_+)_1 \\
\overset{\text{inc}}{\hookrightarrow} \overset q{\vee} S^{n+2m} \wedge |JX|_+ \\
s \mapsto (s,s_2)
\end{gather*}

$z'_i = (\overset{q-1}{\underset{j=i}{\prod}} - \imath(x_i)) \imath(s_1,z) \in \pm|JX_+| \wedge S^{n+m} \wedge |Z|_+ \cong \pm S^{n+m} \wedge |JX|_+ \wedge |Z_+|$, and the product of reduced elementary expansions in (\ref{eqn:2.3.11}) is viewed as an element of $M^n_{q,1}$.  It is straightforward to verify that the
diagram
\define\xa{|X^{q-1} \times (S^m \wedge |Z|_+ \wedge S^m)|}
\define\xb{|M^n_{q,1} \wedge  M^n_{q,2}|}
\define\arr1{|\varphi^0_{m,n,q}\circ f^0_{m,n,q}|} 
\define\xc{|X^{q-1} \times \Sigma . Y|}
\define\xd{|M^n_q(|\overline F_1(X, \Sigma . Y)|)|}
$$
\diagram
|X^{q-1} \times (S^m \wedge |Z|_+ \wedge S^m)|\rto^(.62){\overline\rho^3_{q,6}} \dto_{\simeq} &  |M^n_{q,1} 
\wedge  M^n_{q,2}| \dto^{|\varphi^0_{m,n,q}\circ f^0_{m,n,q}|} \\
|X^{q-1} \times \Sigma . Y|\rto^(.45){\overline\rho^2_{q,6}}  & |M^n_q(|\overline F_1(X, \Sigma . Y)|)|
\enddiagram
$$ 
is canonically homotopy-commutative.  So taking $\overline\rho_{q,6}$ to be $\overline\rho^3_{q,6}$ provides the necessary lift in order to evaluate $\overline{Tr}_X(Y)$.  This evaluation is achieved, according to Theorem \ref{thm:2.2.5}, by switching the terms in (\ref{eqn:2.3.11}) and composing.
Since $\imath^1$ just involves the standard inclusion to the first factor in the wedge, we get 
\[
(\overline e_{11}(z'_1)\cdot\dots\cdot \overline e_{q1}(z'_q)) \circ \imath^1(s_2) = \overline e_{11}(z'_1)\circ \imath^1(s_2)
\]
 which implies that $\overline\pi_p \circ g^0_{m,n,q} \circ \overline\rho^3_{q,6}$ is canonically null-homotopic for $p > q$, and that $\overline\pi_q\circ g^0_{m,n,q} \circ \overline\rho^3_{q,6}$ is the map 
\[
(|X^{q-1}| \wedge S^m \wedge |Z|_+ \wedge S^m) \to \Omega^{n+m} \Sigma^{n+3m} (|X|^{[q-1]} \wedge |Z|_+)
\]
given by 
\[
(x_1,\dots,x_{q-1},s_1,z,s_2) \mapsto (\overset{q-1}{\underset{j=1}{\prod}} - \imath (x_i), s_1,\imath(z),s_2).
\]
Up to reparametrization independent of $X$ and $Z$ this composition is canonically homotopic to $(-1)^{q-1} j_{2m}$,  where $j_{2m}$ is the standard inclusion 
\[
\Sigma^{2m}|X|^{[q-1]}\wedge |Z|_+ \to \Omega^{n+m} \Sigma^{n+3m} (|X|^{[q-1]} \wedge |Z|_+). 
\]
To complete the proof, note that $\overline {Tr}_X (Y)_p\circ\tilde\rho_q (X,Y)$   is induced by a natural transformation of a homogeneous functor of degree $q$ to a  homogeneous functor of degree $p$ (evaluated at $X$), which must be canonically  null-homotopic for $q > p$ as it factors through the $\text{p}^{\text{th}}$ differential of a  $q$-homogeneous functor.
\end{proof}

We may now complete the proof of Theorem A. By the previous theorem, the map
\[
\tilde \rho : \tilde D(X) \to \overline A(\Sigma X)
\]
induces a map on first differentials
\begin{equation}\label{eqn:2.3.14}
(D_1 \tilde \rho)_X (Y) : (D_1 \tilde D)_X (Y) \to (D_1 \overline A\Sigma)_X (Y)
\end{equation}
which is split-injective on homotopy groups for all $X$ and $Y$. In the case that both $X$ and $Y$ are finite complexes, the homotopy groups of both  sides of (\ref{eqn:2.3.14}) are finitely generated; for the right hand side this follows by Theorem \ref{thm:WG}. This implies that  $(D_1 \tilde \rho)_X (Y)$ is an equivalence for all finite $X$ and $Y$. As both functors are also homotopy functors (hence commute up to homotopy with filtered colimits), this implies $(D_1 \tilde \rho)_X (Y)$ is an equivalence for all $Y$ and connected $X$. Finally, applying  Goodwillie's Theorem \ref{thm:conv} at $X = *$ implies that $\tilde \rho$ is an equivalence.
\vskip.2in

The equivalence $\tilde D(X) \overset\rho{\rightarrow} \overline A(\Sigma X)$ is natural with respect to $X$, so that if $f : X\to Y$ is a map of connected simplicial sets there is a homotopy-commutative diagram
$$
\diagram
\tilde D(X)\rto^{\tilde\rho_X}\dto_{\tilde D(f)} &\overline A(\Sigma X)\dto^{\overline A(\Sigma f)} \\
\tilde D(Y)\rto^{\tilde\rho_Y} & \overline A(\Sigma Y)
\enddiagram
$$
It also follows that $\rho$ restricts to yield equivalences 
\[
\overset n{\underset{q=m+1}{\prod}} \tilde D_q(X) = p^m_n\tilde D(X) \overset{p^m_n\tilde\rho(X)}{\longrightarrow} p^m_n\overline A(\Sigma X) 
\] 
natural in $X$ for all $0\le m < n \le \infty$, because $\tilde\rho$ is a natural transformation of homotopy functors and hence commutes with Goodwillie Calculus.  However, it is not true that $\tilde\rho$ or $p^m_n\tilde\rho$ are natural with respect to maps $\Sigma X \overset g{\rightarrow} \Sigma Y$ which do not desuspend up to homotopy.

\end{document}